\documentclass[11pt]{article}
\usepackage{geometry}
\usepackage{amsmath,amsfonts,amssymb, tikz, amsthm}

\usepackage{hyperref}
\usepackage[utf8]{inputenc}
\usepackage{csquotes}
\usepackage{algpseudocode}
\algnewcommand{\Input}[1]{\State \textbf{Input:} #1}
\algnewcommand{\Output}[1]{\State \textbf{Output:} #1}
\usepackage{authblk}

\usepackage{caption}
\usepackage{mathrsfs}
\usepackage{graphicx}
\usepackage[english]{babel}
\usepackage[tableposition=top]{caption} 
\usepackage{array}
\usepackage{algorithmicx}
\usepackage{algpseudocode}
\usepackage{algorithmicx}
\usepackage{algpseudocode}
\usepackage{algorithmicx}
\usepackage{algorithm}
\usepackage{algpseudocode}
\usepackage{algpseudocode}
\usepackage{amsmath}
\usepackage{mathrsfs}
\usepackage{amssymb}
\usepackage{amsfonts}
\usepackage{mathtools}
\usepackage{tikz-cd}
\usepackage[utf8]{inputenc}

\newtheorem{theorem}{Theorem}[section]
\newtheorem{corollary}[theorem]{Corollary}
\newtheorem{lemma}[theorem]{Lemma}

\newtheorem{proposition}[theorem]{Proposition}

\newtheorem{eg}[theorem]{Example}

\theoremstyle{definition}
\newtheorem{definition}[theorem]{Definition}
\newtheorem*{note}{Notation}
\newtheorem*{conv}{Convention}

\newtheorem*{rem}{Remark}
\theoremstyle{remark}

\numberwithin{equation}{section}

\newcommand{\Z}{\mathbb{Z}}

\newcommand{\Q}{\mathbb{Q}}
\newcommand{\C}{\mathbb{C}}
\newcommand{\R}{\mathbb{R}}

\newcommand{\N}{\mathrm{N}}
\newcommand{\T}{\text{Tr}}
\newcommand{\co}{\mathcal{O}}

\newcommand{\keywords}[1]{\noindent\textbf{Keywords:} #1}
\newcommand{\subjclass}[2]{\noindent\textbf{#1 Mathematics Subject Classification:} #2.}

\title{Criterion set for diagonal forms of higher degree}

\author[]{Om Prakash}
\affil[]{Charles University, Faculty of Mathematics and Physics, Department of Algebra, Sokolovsk\'{a} 83, 186 75 Praha 8, Czech Republic}

\affil[]{E-mail: prakash@karlin.mff.cuni.cz. ORCID: 0009-0007-2124-9736}

\begin{document}
\maketitle

\begin{abstract} 
The $290$-Theorem of Bhargava--Hanke completely classifies all universal quadratic forms over the rational integers in terms of a finite criterion set: a positive definite quadratic form is universal if and only if it represents every integer in this finite set. In this article, we study diagonal forms of higher degree over a totally real number field and prove the existence of a unique finite criterion set, minimal with respect to inclusion, for universal diagonal forms of higher degree. As part of the proof, we establish an analog of the asymptotic local–global principle for such forms and develop a local theory for the representation of integers by diagonal forms of higher degree. 

\medskip
	 
\medskip
\subjclass{2020}{11E76, 11E10, 11R04, 11R80}
\end{abstract}

\section{Introduction}

Determining which integers can be represented by a given quadratic form is a classical problem in number theory that has interested mathematicians since ancient times. In this context, a central question is whether a given \emph{positive definite} quadratic form is \emph{universal}, i.e., represents all positive integers. In $1770$, Lagrange showed that a sum of four squares $x^2+y^2+z^2+w^2$ is universal. Since universal quadratic forms exist, it is natural to ask whether we can classify all universal quadratic forms.

It was Ramanujan \cite{ram1} who initiated the classification of quaternary diagonal universal quadratic forms, i.e., forms given  by $ax^2+by^2+cz^2+dw^2$, and  Dickson \cite{dic} completed the classification. This work was subsequently extended by Willerding \cite{wil}, who classified all  universal quadratic forms with integral \emph{Gram matrices} (i.e., \emph{classical} quadratic forms) in four variables. In $1993$, Conway–Schneeberger \cite{sch,bhargava1} proved the $15$-Theorem, characterizing classical universal quadratic forms in any number of variables. Specifically, they showed such a form is universal if and only if it represents $1,2,3,5,6,7,10,14, \text{ and } 15$. Thus, the set consisting of these $9$ integers serves as a finite criterion set for universal classical quadratic forms. These classification results culminated in the $290$-Theorem of Bhargava--Hanke \cite{bhargava2}, which states that a positive definite quadratic form is universal if and only if it represents 
\begin{center}
    $ 1, 2, 3, 5, 6, 7, 10, 13, 14, 15, 17, 19, 21, 22, 23, 26,$\\
    $29, 30, 31, 34, 35, 37, 42, 58, 93, 110, 145, 203,\text{ and }290$.
\end{center}
Moreover, they showed that this list of $29$ integers is minimal. Besides classifying all universal quadratic forms, quadratic forms representing only a given subset of $\Z$ have gained significant attention, especially in relation to the conjectural $451$-Theorem of Rouse \cite{rou} and related results \cite{br,de,dr}.

The study of universal quadratic forms becomes more subtle when one passes from the rational integers to the ring of algebraic integers $\co_K$ inside a totally real number field $K$. A quadratic form over $K$, i.e., with coefficients of the form lying in $\co_K$, is called universal if it is \emph{totally positive definite} and represents all \emph{totally positive integers} $\co_K^+$. Maa\ss{} \cite{maa} showed that a sum of three squares is universal over $\Q(\sqrt{5})$. In fact, this is the only non-trivial totally real number field over which a sum of squares is universal, thanks to a result of Siegel \cite{si1}. Nevertheless, universal quadratic forms exist over every totally real number field, as follows from the asymptotic local–global principle of Hsia--Kitaoka--Kneser \cite{hkk}. They proved that, provided the quadratic form has at least $5$ variables, representation of totally positive integers with sufficiently large norm satisfies a local--global principle. 

While the asymptotic local--global principle implies the existence of universal quadratic forms over $K$, this principle yields no information about the size of the \emph{minimal rank}. In particular, Kitaoka's conjecture--that only finitely many totally real number fields admit a universal ternary quadratic form--remains open, although some partial results are known for families of fields \cite{ckr, ek,ktz,kkk,kkpyz, kk}, and for the fixed degree fields \cite{ky}.

\medskip

As for criterion sets for quadratic forms over number fields, the situation remains largely mysterious. Chan--Oh \cite{co} established the existence of finite criterion sets for quadratic forms over totally real number fields, however, their result does not provide any information about the structure or size of criterion sets. The first explicit criterion set over number field was given by Lee \cite{lee}, who showed that a classical quadratic form over $\Q(\sqrt{5})$ is universal if and only if it represents $7$ specific elements. Recently, several results about criterion sets have been established, with some more in preparation \cite{kmvy}. Notably, Kala--Kr\'{a}sensk\'{y}--Romeo \cite{kkr} proved that for every totally real number field $K$, a unique minimal criterion set exists. Moreover, they also showed some specific elements that must be there in the criterion set (e.g., all squarefree \emph{indecomposables}). However, their results do not provide an explicit criterion set, even over a single $K$. But recently, Kr\'{a}sensk\'{y}--Romeo \cite{kr} have conjectured the expected criterion set over certain real quadratic fields, and discussed how to obtain other similar conjectures by adapting the standard method of \emph{escalation} to number fields. 

On the other hand, a lot is known about criterion sets for \emph{$n$-universal} forms, where a positive definite quadratic form over $\Z$ is called $n$-universal if it represents all positive definite quadratic forms of \emph{rank} $n$. Beyond the general finiteness result for criterion sets of Kim--Kim--Oh \cite{kko} and Chan--Oh \cite{co} in general, Kim--Kim--Oh \cite{kko2} have found a six element criterion set for $2$-universality, and Kominers \cite{kom} proved its uniqueness. Also, for $8$-universality, the unique criterion set is known \cite{bko, kom2}. In contrast, Kim--Lee--Oh \cite{klo} showed that criterion sets are not unique for $n\geq 9$. The non-uniqueness of criterion sets was first observed by Elkies--Kane--Kominers \cite{ekk}. Additionally, for forms over local fields, criterion sets have been better understood \cite{be1,hh,hhx,xz}, with related results for Hermitian forms \cite{kkp}.

\medskip

In this article, we study the universality of $m$-ic forms, i.e., homogeneous polynomials of degree $m> 2$, with a particular focus on the key question of whether such forms can be classified by finite criterion sets. Partly, the study of universal $m$-ic forms is motivated by the classical Waring's problem, which dates back to $1770$. Hilbert's solution to Waring's problem showed that there is a finite constant $g(m)$ such that every positive integer can be expressed as a sum of at most $g(m)$ $m$th powers, providing examples of universal $m$-ic forms over $\Z$. While most of the recent works in \emph{circle method} focus on $G(m)$, the number of $m$th powers required to represent all sufficiently large integers, the constant $g(m)$ is known for almost all $m$ \cite{vw}.  An $m$-ic form $Q$ over $\Z$ is called positive definite if $Q(x)>0$ for all $x\in \R\setminus\{0\}$, which immediately forces $m$ to be even, thanks to the homogeneity. We say that an $m$-ic form $Q$ is universal if it is positive definite and represents all positive integers. Although $m$-ic forms appear to be a natural generalization of quadratic forms, their structure is fundamentally different. Consequently, most of the techniques that work for quadratic forms fail for $m$-ic forms, making their study even more difficult.

As one might expect, the study of $m$-ic forms over totally real number fields is even more enigmatic. Only recently, Kala and the author established that universal $m$-ic forms exist over every totally real number field; however, unlike the quadratic case, they cannot be characterized by finite criterion sets, even over $\Z$ \cite{kp}. More precisely, we showed that the criterion set for universal $m$-ic forms is exactly the set of all $m$th powerfree integers (up to multiplication by $m$th powers of units), which is infinite. Some further results about universal $m$-ic forms over number fields were obtained by the author in \cite{pra}. Nevertheless, significant progress has been made in the algebraic theory of $m$-ic forms. This includes a Witt theory \cite{hr} for $m$-ic forms, as well as extensions of several fundamental results from the algebraic theory of quadratic forms to this broader setting; see \cite{pumpl4, morandi, pumpl3,  pumpl2, pumpl1}.

Note that an infinite criterion set always exists, as we can simply take the entire set of (totally) positive integers. Throughout the paper, by criterion sets, we mean finite criterion sets unless stated otherwise. To determine a criterion set, if it exists, one needs a way to enumerate all ``proper" universal forms. For quadratic forms, this is typically done using the so-called {escalation} procedure \cite{kr}. However, when one attempts to apply this procedure to $m$-ic forms, a fundamental difficulty arises—-at each stage of this process, infinitely many forms may occur, leading to the infiniteness of criterion sets, in contrast to the quadratic case, where only finitely many forms arise at each step. The primary source of this difficulty can be seen already from the simultaneous representation of two positive integers: Given positive integers $a,b$, there are infinitely many positive definite $m$-ic forms representing both $a$ and $b$, namely $ax^m+cx^2y^{m-2}+by^m$ for $c\in \Z_{\geq 0}$. This does not occur for quadratic forms, thanks to the Cauchy--Schwarz inequality. 

In this paper, we study diagonal $m$-ic forms, where such a situation does not arise because these forms do not have cross-terms. Thus, the restriction to diagonal forms is not merely a choice (e.g., for simplicity) but is rather forced by the circumstances. From this point onward, we only consider totally positive definite diagonal $m$-ic forms over a totally real number field $K$. Our goal is to classify all universal diagonal $m$-ic forms over a totally real number field $K$ in analogy to the $290$-Theorem. Since all our results are immediate for diagonal quadratic forms, we henceforth assume $m\ge 2$ is even. The main result of this paper is the following:

\begin{theorem}\label{main2}
    Let $m\geq 2$ be an even integer and $K$ be a totally real number field. There exists a unique finite criterion set for universal diagonal $m$-ic forms over $K$ that is minimal with respect to inclusion. In fact, this set contains precisely all critical elements.
\end{theorem}

We prove this in Theorem \ref{uniqueness}. In order to prove the uniqueness and minimality, we first need to prove the existence of criterion sets. In the escalation procedure for diagonal forms, which we develop in Section \ref{S3}, each step has only finitely many forms (Lemma \ref{finitenessof truant}), and its termination provides a criterion set (Proposition \ref{lem:finitecriterionset}). However, proving that this process terminates is the difficult part of establishing the existence of criterion sets. In the quadratic case, this is accomplished using the asymptotic local--global principle of Hsia--Kitaoka-Kneser \cite{hkk}.

\begin{rem}
    We once again stress that Theorem \ref{main2} is not true without the diagonal assumption, since in that case, the criterion set is infinite \cite{kp}. Also, the case $m=2$ is already covered in \cite{kkr}.
\end{rem}

The failure of the Hasse principle (local--global principle) for higher degree forms suggests that an asymptotic local--global principle may not hold. However, such failure mainly arises from global obstructions, such as \emph{Brauer–Manin obstructions}, especially in fewer variables. Thus, for forms with large rank, one may hope that such obstructions are less likely to arise; that makes the asymptotic local--global principle more plausible. We prove the following asymptotic local--global principle for diagonal $m$-ic forms in Corollary \ref{cor:alg}. 

\begin{corollary}\label{cor:main}
    Let $m\geq 2$ be an even integer, $K$ be a totally real number field of degree $d$, and $Q(x)=\sum_{i=1}^na_ix_i^m$ be a totally positive definite diagonal $m$-ic form in $n$ variables over $K$ with $n>d(d+1)(m-1)2^{m-1}$. Assume that there exists an integer $t\geq 0$ depending only on $m,K,Q$, such that for every $\alpha\in\co_K^+$ and for every prime ideal $\mathfrak{p}\subset\co_K$ there exists $x\in\co_\mathfrak{p}^n$ with $Q(x)=\alpha$ and 
    \[
    \min_{1\leq i\leq n}v_\mathfrak{p}(ma_ix_i^{m-1})\leq t.
    \]
    
     Then, there exists a constant $C>0$ depending only on $m,K,Q,t$, such that every $\alpha\in \co_K^+$ with $\N(\alpha)>C$ is represented by $Q$.
\end{corollary}

The proof of this occupies a substantial portion of the paper. We call this a corollary because we obtain it from a theorem of Birch \cite[Theorem $1$]{bir}, rather than directly applying the \emph{circle method} over the number field. A direct application of the circle method would likely give a better bound on the number of variables (e.g., $n\geq 2^m+1$), but it would require us to develop substantially more analytic machinery than is needed for our purposes. The current bound provided by Birch's theorem is sufficient; this approach allows us to isolate the arguments that are most relevant in the context of this paper. Let us also highlight that the circle method argument would not avoid the arithmetic difficulties that arise here. For instance, the \emph{singular integral} and the \emph{singular series} require a separate treatment. 

\medskip

Applying Birch's theorem in our setting is not straightforward; it requires several additional ingredients. To prove Corollary \ref{cor:main}, our strategy is as follows. First, from a diagonal $m$-ic form over $K$, we obtain a descent system over $\Z$ using the \emph{Weil restriction} (e.g., see \cite{poo}, \cite{bps}). Using a special integral basis for $K$, constructed in Section \ref{S4} with tools from the geometry of numbers, we show that the problem of representing $\alpha\in\co_K^+$ by a diagonal $m$-ic form can be translated into a problem over $\Z$. More precisely, it is equivalent to representing the coordinates of $\alpha$, with respect to this special integral basis, by a system of $m$-ic forms (not necessarily diagonal) over $\Z$. Moreover, under this translation, large norm elements translate to elements with large coordinates. Now, to study the representation of coordinate vectors by a system of $m$-ic forms, we use Birch's theorem \cite[Theorem $1$]{bir}, which provides an asymptotic formula for counting such coordinate vectors by the system of $m$-ic forms.

It can easily happen that the main term of Birch's asymptotic is zero. Hence, the asymptotic cannot be used to prove the representation of all large coordinate vectors. In our setting, establishing the non-vanishing of the {singular integral} is relatively straightforward. The main difficulty, however, lies in treating the {singular series}, which requires a new strategy. Birch's sufficient condition states that if the variety generated by the system of $m$-ic forms has a \emph{non-singular} $mod \, p$ point,  then the singular series is uniformly bounded away from zero; this condition may not be satisfied in our case because, for primes dividing $m$, the rank of the Jacobian can easily be zero. We prove a different sufficient condition in Section \ref{S6}. In particular, the assumption stated in Corollary \ref{cor:main} implies the existence of a uniform positive bound on the singular series. As stated, this condition may look strange. However, we verify in Section \ref{S8} that the assumption holds for forms arising in the escalation procedure, provided they have at least $d(d+1)(m-1)2^{m-1}$ variables. In particular, the assumption holds for \emph{primitively locally universal} forms, i.e., $Q(x)=\alpha$ with some $x_i\in\co_K^\times$.

As in the case of quadratic forms, Corollary \ref{cor:main} provides a new proof of the existence of a universal diagonal $m$-ic form over every totally real number field \cite[Theorem $5.4$]{kp}.

Using Corollary \ref{cor:main}, we show in Theorem \ref{termination} that the escalation procedure terminates. This immediately supplies the existence of finite criterion sets for universal diagonal $m$-ic forms (see Corollary \ref{cor:9.4}). Then, using the strategy from our Theorem \ref{termination}, we prove that the set of all \emph{critical} elements (see Definition \ref{def:critical}), which by definition must be in every criterion set, is indeed a criterion set.  

\medskip

Let us emphasize that, although universal diagonal $m$-ic forms are characterized by a unique finite criterion set (Corollary \ref{cor:main}), we do not yet have an explicit description of this set, even over $\Z$, nor  do we have a bound on the norms of its elements. A forthcoming work of Kala and the author discusses the explicit criterion set for universal diagonal $m$-ic forms over $\Z$ conjecturally.  

As in the work of Kala--Kr\'{a}sensk\'{y}--Romeo \cite{kkr}, we also prove some examples of elements that must be critical, so they lie in every criterion set. For example, the set of all \emph{$m$th powerfree} indecomposables must lie in every criterion set. In particular, at least over real quadratic fields, using Blomer--Kala's \cite{bk} description of indecomposables, one can count these $m$th powerfree indecomposables to obtain a lower bound on the criterion set. However, we do not pursue that here, leaving it for future work. 

While most of the paper is concerned with the representation of global integers, i.e., elements of $\co_K^+$, we also establish several interesting results concerning the representation of local integers along the way. We believe that these results will also be very useful in the future. 

\section*{Acknowledgments}

The author is very grateful to V\'{\i}t\v{e}zslav Kala, Jakub  Kr\'{a}sensk\'{y}, and Damaris Schindler for their helpful suggestions on earlier drafts, as well as to Subham Roy, Robin Visser, and Pavlo Yatsyna for several interesting and useful discussions. The author was supported by {Czech Science Foundation grant} 26-20514S, {Charles University programmes}  PRIMUS/25/SCI/017, SVV-2023-260721, and {GAUK project No.} 236225. ChatGPT free was used for the reformulation of sentences and for checking grammatical errors.

\section{Preliminaries}

In this section, we collect some necessary definitions and concepts that we use throughout the paper. 

\subsection{Number fields} Throughout the article, $K$ will be a totally real number field of degree $d$, i.e., all its embeddings $\sigma_1,\sigma_2,\dots,\sigma_d: K\rightarrow \C$ have images in $\R$. We denote by $\co_K$ and $\co_K^\times$ the ring of integers inside $K$ and the group of units, respectively. An \emph{integral basis} (or a \emph{$\Z$-basis}) for $K$ is a set 
$\{\omega_1,\omega_2,\dots,\omega_d\}\subset\co_K$ such that 
\[
\co_K=\Z\omega_1+\Z\omega_2+\cdots+\Z\omega_d.
\]
Thus, every $\alpha\in\co_K$ can uniquely be expressed as $\alpha=\alpha_1\omega_1+\cdots+\alpha_d\omega_d$ with $\alpha_i\in \Z$. We say that $(\alpha_1,\alpha_2,\dots,\alpha_d)\in \Z^d$ is the \emph{coordinate} of $\alpha$ with respect to an integral basis $(\omega_1,\dots,\omega_d)$ if $\alpha=\alpha_1\omega_1+\cdots+\alpha_d\omega_d$. For $\alpha\in K$, its norm and trace are defined as follows: 
\[
\N(\alpha)=\prod_{i=1}^d \sigma_i(\alpha), \quad \T(\alpha)=\sum_{i=1}^d \sigma_i(\alpha).
\]

Throughout, we denote by $\Delta=(\sigma_i(\omega_j))_{1\leq i,j\leq d}$ the invertible matrix whose determinant squared is the \emph{discriminant} of $K$. Note that the discriminant is an invariant and does not depend on the choice of an integral basis. Now, recall that the trace  defines a non-degenerate symmetric $\Q$-bilinear pairing 
\[
\langle \cdot,\cdot\rangle: K\times K\rightarrow \Q: \langle x,y\rangle\mapsto \T(xy).
\]
Throughout this text, we will refer to this bilinear pairing as the \emph{trace-pairing}. The \emph{codifferent}, denoted by $\co^\vee_K$, is the dual of $\co_K$ with respect to trace-pairing; it is defined as \[
\co^\vee_K=\{x\in K\mid \T(x\co_K)\subseteq \Z\}.
\] 
It is a fractional ideal containing $\co_K$ and a free $\Z$-module of rank $d$. Its dual under the trace-pairing is $\co_K$, i.e., 
\[
\{y\in K\mid \T(x\co^\vee_K)\subseteq \Z\}=\co_K.
\]
Let $\{\omega_1',\dots,\omega_d'\}$ be a $\Z$-basis of $\co^\vee_K$. Define the dual basis $\{\omega_1,\dots,\omega_d\}$ by 
\[
\T(\omega_i'\omega_j)=\delta_{ij}=\begin{cases}
        1, \text{ if } i=j\\
        0, \text{ if } i\neq j.
    \end{cases}
\]
Then, $\{\omega_1,\dots,\omega_d\}$ is a $\Z$-basis for $\co_K$.

We say that an element $\alpha\in K$ is \emph{totally positive} if $\sigma_i(\alpha)>0$ for all $1\leq i\leq d$. For a subset $S\subseteq K$, $S^+$ denotes the set of all totally positive elements in $S$.

For $\alpha, \beta \in K$, we say that $\alpha$ is totally greater than $\beta$, denoted by $\alpha \succ \beta$ if $\alpha-\beta$ is totally positive; equivalently, $\sigma_i(\alpha)>\sigma_i(\beta)$ for all $1\leq i \leq d$; moreover, $\alpha \succeq \beta$ denotes that $\alpha \succ \beta$ or $\alpha = \beta$. A totally positive integer $\alpha\in \co_K^+$ is called \emph{indecomposable} if it cannot be written as a sum of two totally positive integers, i.e., $\alpha\neq \beta+\gamma$ for all $\beta,\gamma\in \co_K^+$. 

For $\Z$-modules (abelian groups) $A,B$, the \emph{tensor product} $A\otimes_
\Z B$ is a $\Z$-module together with a $\Z$-bilinear form 
\[
\otimes: A\times B \rightarrow A\otimes_\Z B: (a,b)\mapsto a\otimes b,
\]
such that for every $\Z$-module $C$ and every $\Z$-bilinear map $f: A\times B\rightarrow C$, there exists a unique homomorphism $\tilde{f}: A\otimes_\Z B\rightarrow C$ such that $\tilde{f}(a\otimes b)=f(a,b)$. If $A$ and $B$ are \emph{free} with bases $\{e_i\}$ and $\{f_j\}$ respectively, then $A\otimes_\Z B$ is free with a basis $\{e_i\otimes f_j\}$.

\subsection[Diagonal m-ic forms]{Diagonal $m$-ic forms} Let $m$ be a positive integer $\geq2$. A diagonal $m$-ic form $Q$ over $K$ of rank $n$ is given by the expression 
\[
Q(x)=Q(x_1,\dots,x_n)=\sum_{i=1}^n a_ix_i^m=:\langle a_1,a_2,\dots,a_n\rangle,
\]
where $a_i\in \co_K$. Here and throughout, we use the rank of $Q$ as a synonym for the number of variables. 

Let $Q_1,Q_2$ be diagonal $m$-ic over $K$ of rank $n_1,n_2$. The \emph{orthogonal sum} of $Q_1$ and $Q_2$ is a diagonal $m$-ic form $Q=Q_1\perp Q_2$ given by 
\[
Q(x_1,\dots,x_{n_1},y_1,\dots,y_{n_2})=Q_1(x_1,\dots,x_{n_1})+Q_2(y_1,\dots,y_{n_2}).
\]

Two diagonal $m$-ic forms $Q_1(x)=\sum_{i=1}^na_ix_i^m$, $Q_2(y)=\sum_{i=1}^nb_iy_i^m$ of rank $n$ over $K$ are called \emph{isometric} (or \emph{equivalent}) if there is an invertible matrix $P=(p_{i,j})_{1\leq i,j\leq n}\in GL_n(\co_K)$ such that 
\[
\sum_{i=1}^na_i\left(\sum_{j=1}^np_{i,j}y_j \right)^m=\sum_{j=1}^nb_jy_j^m,
\]
i.e., $Q_2$ is obtained from $Q_1$ by an invertible linear change of variables. In this paper, we will consider forms up to equivalence, as equivalent forms essentially represent the same elements.

An $m$-ic diagonal form $Q$ over $K$ is called  \emph{totally positive definite} if, for all embeddings $\sigma_i: K\rightarrow \R$
\[
\sigma_i(Q)(x)=\sum_{i=1}^n \sigma(a_i)x_i^m
\]
is positive definite over $\R$, i.e., $\sigma_i(Q)(x)>0$ for all $x\in\R^n\setminus \{0\}$; equivalently, we can state that $Q(x)$ is totally positive for all $x\neq 0$. Since $Q(-x)=(-1)^mQ(x)$ holds for all $x\in\co_K^n$, for a totally positive definite form, we must have $m\in 2\Z_{\geq1}$. Thus, unless stated otherwise, $m$ will always be an even positive integer $\geq 2$. Consequently, $Q$ being totally positive definite is equivalent to $a_i\succ 0$ for all $1\leq i\leq n$.

We say that $\alpha\in \co_K$ is \emph{represented} by a diagonal $m$-ic form $Q$ if there exists $x\in\co_K^n$ such that $Q(x)=\alpha$. A totally positive definite diagonal $m$-ic form $Q$ is \emph{universal} if it represents every element of $\co_K^+$.

Let us denote by $\co_K^+/\co_K^{\times m}$ the set of classes of elements of $\co_K^+$ up to multiplication by $m$th powers of units. Then, by homogeneity, a diagonal $m$-ic form over $K$ represents $\alpha\in\co_K$ if and only if $\alpha\varepsilon^m$ is represented by $Q$ for all $\varepsilon\in \co_K^\times$. Thus, it is enough to study the representation of elements of $\co_K^+/\co_K^{\times m}$. Note that the norm is well-defined on $\co_K^+/\co_K^{\times m}$, i.e., it does not depend on the choice of the representative. 

Let $\beta\in \co_K^+/\co_K^{\times m}$ be a class. By the $m$-ic form $\langle \beta\rangle$, we refer to the form $\beta x^m$. This is well-defined, as we are considering forms up to equivalence because for any other representative $\beta\varepsilon^m$, the two forms $\beta\varepsilon^m x^m$ and $\beta x^m$ are equivalent in the above sense.

\subsection{Local theory} A \emph{discrete valuation ring } is a principal ideal domain that has a unique non-zero prime ideal. Therefore, it has a unique maximal ideal, and hence it is a local ring. Let $R$ be a discrete valuation ring with maximal ideal $\mathfrak{m}$, which is, by definition, a principal ideal. A generator of $\mathfrak{m}$ is called a \emph{uniformizer}, denoted by $\pi$. Every non-zero element of $\co$ can be uniquely written as $\pi^k u$ with $k\in\Z_{\geq 0}$ and $u\in R^\times$, where $R^\times$ is the group of units in $R$. Every non-zero element of the fraction field $K=\operatorname{Frac}(R)$ can be  written uniquely as $\pi^k u$ with $k\in\Z$ and $u\in R^\times$. We can define a \emph{valuation} $v_\mathfrak{m}:K^\times\rightarrow \Z$ given by $v_\mathfrak{m}(\pi^ku)=k$, which is a surjective homomorphism satisfying: for all $x,y\in K^\times $ such that $x+y\neq 0$, $v_\mathfrak{m}(x+y)\geq \min (v_\mathfrak{m}(x),v_\mathfrak{m}(y))$, with equality if $v_\mathfrak{m}(x)\neq v_\mathfrak{m}(y)$. We can extend the valuation to the whole $K$ by setting $v_\mathfrak{m}(0)=\infty$. The \emph{valuation ring} of $v_\mathfrak{m}$ is 
\[
\{x\in K^\times\mid v_\mathfrak{m}(x)\geq 0\}\cup \{0\}=R.
\]
In particular, for a totally real number field $K$, we can recover the valuation ring $\co_K$ as the set of elements with nonnegative valuation.
We say that a discrete valuation ring is \emph{complete} if it is complete with respect to the $\mathfrak{m}$-adic topology. The field $\co_K/\mathfrak{m}\co_K$ is called the \emph{residue field}.

Let $K$ be a number field and $\mathfrak{p}\subset\co_K$ a non-zero prime ideal. The \emph{localization} of $\co_K$ at $\mathfrak{p}$
\[
\co_{(\mathfrak{p})}=\left\{\frac{a}{s}\mid a\in \co_K, s\in \co_K\setminus \mathfrak{p} \right\}
\]
is a discrete valuation ring with maximal ideal $\mathfrak{p}\co_{(\mathfrak{p})}$. The completion of $\co_{(\mathfrak{p})}$ with respect to this maximal ideal is the inverse limit 
\[
\co_\mathfrak{p}=\varprojlim_i \co_{(\mathfrak{p})}/\mathfrak{p}^i\co_{(\mathfrak{p})}.
\]
The ring $\co_\mathfrak{p}$ is a complete discrete valuation ring, with the maximal ideal generated by $\mathfrak{p}$. The fraction field $K_\mathfrak{p}=\operatorname{Frac}(\co_\mathfrak{p})$ is the completion of $K$ at $\mathfrak{p}$ with respect to the $\mathfrak{p}$-adic topology. We denote the $\mathfrak{p}$-adic valuation on $K_\mathfrak{p}$ by $v_\mathfrak{p}$, and throughout, we will work with a fixed \emph{uniformizer} (i.e., a generator of the maximal ideal) $\pi_\mathfrak{p}$. The residue field is $\co_\mathfrak{p}/\mathfrak{p}$, a finite field of characteristic $p$, where $p$ is a rational prime lying below $\mathfrak{p}$, denoted by $\mathfrak{p}|p$. Its cardinality is $\N(\mathfrak{p})=p^{f(\mathfrak{p}|p)}$, where $f(\mathfrak{p}|p)$ is \emph{inertia degree}.

We will use the following generalized Hensel's lemma, from which the usual Hensel's lemma for a discrete valuation ring follows.

\begin{lemma}[{Generalized Hensel lemma}]\label{genhensel}
    Let $R$ be a complete discrete valuation ring with maximal ideal $\mathfrak{m}$, valuation $v_\mathfrak{m}$, $f_1,f_2,\dots,f_d\in R[x_1,x_2,\dots,x_d]$, $a\in R^d$, and $J=(\partial f_i/\partial x_j)(a)\in M_{d\times d}(R)$. Suppose $\det J\neq 0$ and write $\tau=v_\mathfrak{m}(\det J)$. If $v_\mathfrak{m}(f_i(a))\geq 2\tau+1$ for all $1\leq i\leq d$, then there exists a unique $b\in R^d$ such that $f_i(b)=0$ for all $1\leq i\leq d$ and $v_\mathfrak{m}(b-a)=\min_{1\leq i\leq d }v_\mathfrak{m}(b_i-a_i)\geq \tau+1$.
\end{lemma}

Next, we need the following standard approximation property, whose proof we include only for completeness.  

\begin{lemma}\label{approximation}
    For every integer $n\geq 1$, the natural map 
    \[
    \co_K\rightarrow \co_\mathfrak{p}/\mathfrak{p}^n\co_\mathfrak{p}
    \]
    is an isomorphism. In particular, for $\alpha\in \co_\mathfrak{p}$, there exists a $\beta\in \co_K$ such that $\alpha\equiv \beta\pmod {\mathfrak{p}^n\co_\mathfrak{p}}$.
\end{lemma}

\begin{proof}
    For each fixed $n\geq 1$, consider the natural map $\co_K\rightarrow \co_{(\mathfrak{p})}/\mathfrak{p}^n\co_{(\mathfrak{p})}$. The kernel of this map is $\mathfrak{p}^n\co_{(\mathfrak{p})}\cap \co_K=\mathfrak{p}^n$, so the map is injective. To prove surjective, pick an element $a/s$ modulo $\mathfrak{p}^n$ in $\co_{(\mathfrak{p})}/\mathfrak{p}^n\co_{(\mathfrak{p})}$. As $s$ is invertible modulo $\mathfrak{p}^n$, there exists $t\in \co_K$ such that $st\equiv 1 \pmod {\mathfrak{p}^n}$. Then, $a/s \equiv at \pmod {\mathfrak{p}^n}$, and $at\in \co_K$. Thus, the map is surjective, so $\co_K/\mathfrak{p}^n\co_K\simeq \co_{(\mathfrak{p})}/\mathfrak{p}^n\co_{(\mathfrak{p})}$. 

    By passing to the completion, we get an isomorphism $\co_{(\mathfrak{p})}/\mathfrak{p}^n\co_{(\mathfrak{p})} \rightarrow \co_\mathfrak{p}/\mathfrak{p}^n\co_\mathfrak{p}$. Therefore, the composition of the two isomorphisms yields the required isomorphism. 
\end{proof}

\begin{note}
    Throughout the paper, we will use the standard notations from analytic number theory--for nonnegative functions $f(x),g(x)$, we write $f(x)=O(g(x))$ or $f(x)\ll g(x)$ if there is a constant $C$ such that $f(x)\leq C g(x)$ holds for all sufficiently large $x$. Further, we will use $f(x)=O_{\ell}(g(x))$ or $f(x)\ll_{\ell} g(x)$ to indicate that the implied constant $C$ depends on parameter(s) $\ell$.
\end{note}

\section{Escalation of diagonal forms}\label{S3}

In this section, we develop the escalation algorithm for diagonal forms of higher degree over a totally real number field $K$. Throughout this section, $m\geq 2$ is an even integer, and $K$ is a totally real number field. By a criterion set, we always mean a finite criterion set.

The theory of escalations for diagonal quadratic $(m=2)$ forms over number fields was independently developed in \cite[Section $4$]{kr}. Our definition of an escalation is simpler than their definition of \enquote{pseudoescalation}, and in general, our presentations differ in many details. The difference between the definitions is essentially explained in \cite[Proposition $4.5$]{kr}.

\begin{definition}
    Let $K$ be a totally real number field and $m\geq 2$ be an even integer. We say that a finite set $\mathcal{C}_K^{(m)}\subset \co_K^+/\co_K^{\times m}$ is a criterion set for diagonal $m$-ic forms over $K$ if for every totally positive definite diagonal $m$-ic form $Q$ over $K$ the following statements are equivalent
    \begin{enumerate}
        \item $Q$ is universal,
        \item $Q$ represents all elements of $\mathcal{C}_K^{(m)}$.
    \end{enumerate}
\end{definition}

We are writing $\mathcal{C}_K^{(m)}$ to emphasize that the criterion sets depend on the fixed number field $K$ and $m$.

\begin{definition}
    Let $Q$ be a totally positive definite $m$-ic form over $K$. We say that $\alpha\in \co_K^+/\co_K^{\times m}$ is a truant of $Q$ if $Q$ does not represent $\alpha$, and $Q$ represents all $\beta\in \co_K^+/\co_K^{\times m}$ with $\N(\beta)<\N(\alpha)$. 
\end{definition}

It is obvious from the definition that a truant of $Q$ need not be unique. For example, every totally positive unit $\varepsilon\notin \co_K^{\times m}$ is a truant of the form $Q(x)=x^m$. On the other hand, all truants always have the same norm.

\begin{definition}\label{def:critical}
    We say that an element $\alpha\in \co_K^+/\co_K^{\times m}$ is a critical element if there exists a totally positive definite diagonal $m$-ic form $Q$ over $K$ such that $Q$ does not represent $\alpha$, and $Q$ represents all $(\co_K^+/\co_K^{\times m})\setminus \{\alpha\}$. Throughout the paper, we denote by $\overline{\mathcal{C}}_K^{(m)}$ the set of all critical elements. 
\end{definition}

Observe that, given one criterion set $\mathcal{C}_K^{(m)}$, we can obtain infinitely many other criterion sets--for instance, by adding finitely many totally positive integers. However, critical elements lie in every criterion set; thus, $\overline{\mathcal{C}}_K^{(m)}$ is the unique minimal criterion set.

Let us now define the escalation of a diagonal form. Throughout this paper, we denote by $Q_0$ the empty form of rank $0$. This is an $m$-ic form that represents only $0$.

\begin{conv}
    For simplicity, we consider the empty form $Q_0$ as a totally positive definite diagonal $m$-ic form.
\end{conv}

\begin{definition}
   Let $Q$ be a totally positive definite diagonal $m$-ic form over $K$. Assume that $Q$ is non-universal with a truant $\alpha$. An escalation of $Q$ at $\alpha$ is a diagonal $m$-ic form $Q'=Q\perp \langle \beta\rangle$ with $\beta \in \co_K^+/\co_K^{\times m}$ such that $Q'$ represents $\alpha$. If $Q$ is universal, then its only escalation is $Q$ itself. 
\end{definition}

Note that there could be several escalations at the chosen truant. For example, the form $\langle 1\rangle$ over $K=\Q$ has truant $2$, and there are two possible escalations of this form at $2$, namely $\langle 1,1\rangle$ and $\langle 1,2\rangle$. Before we proceed further, for simplicity, we define several informal notions that we will use throughout the paper.

\medskip

\textbf{New Notation.} Let $Q$ be a totally positive definite diagonal $m$-ic form over $K$ of rank $n$. Define:

\begin{itemize}
    \item $\operatorname{Rep}(Q)=\{\alpha\in \co_K^+/\co_K^{\times m}\mid \exists\, x\in \co_K^n \text{ such that }Q(x)=\alpha \}\cup \{0\}$.
    \item $\operatorname{Unrep}(Q)=(\co_K^+/\co_K^{\times m})\setminus \operatorname{Rep}(Q)$.
    \item $\min \mathrm{N}(\operatorname{Unrep}(Q))=\min\{\N(\alpha)\mid \alpha\in \operatorname{Unrep}(Q)\}$.
    \item $T(Q)=\{\alpha\in \co_K^+/\co_K^{\times m}\mid \alpha \text{ is a truant of } Q\}$.
    \item If $Q$ is non-universal with a truant $\alpha$, then we denote by
    $$E(Q,\alpha)=\{\beta\in \co_K^+/\co_K^{\times m}\mid Q\perp\langle \beta\rangle \text{ represents } \alpha\},$$
    and by 
    $$E(Q)=\bigcup_{\alpha\in T(Q)} E(Q,\alpha).$$
\end{itemize}

Let us remind the reader that for a class $\beta\in \co
_K^+/\co_K^{\times m}$, the forms $\beta x^m$ and $\beta\varepsilon^m x^m$ are equivalent, so they represent the same elements. Therefore, it essentially does not matter which representative we choose. 

\begin{lemma}\label{lem:3.6}
    Let $m\geq 2$ be an even integer, $K$ a totally real number field, and $Q$ a totally positive definite non-universal diagonal $m$-ic form over $K$ of rank $n$ with a truant $\alpha$. Then: 
    \begin{enumerate}
        \item A class $\beta\in E(Q,\alpha)$ if and only if there exists $y\in \co_K\setminus\{0\}$ such that $\beta y^m\preceq \alpha$.

        \item For $\beta\in E(Q,\alpha)$, the rank of $Q\perp \langle\beta\rangle$ is $n+1$, and $\N(\beta)\leq \mathrm{N}(\alpha)$.
    \end{enumerate}
\end{lemma}

\begin{proof}
    Let us first prove (1). Suppose $\beta\in E(Q,\alpha)$, i.e., $Q\perp \langle \beta \rangle$ represents $\alpha$. Then, by definition, there exist $x\in \co_K^n$ and $y\in \co_K$ such that 
        \[
        \alpha=Q(x)+\beta y^m.
        \]
    If $y=0$, then we get a contradiction to the fact that $\alpha$ is a truant of $Q$. Therefore, $y\neq 0$, so $\alpha-\beta y^m=Q(x)\succeq 0$. Hence, we have $\beta y^m\preceq \alpha$. Conversely, suppose that there exists $y\in \co_K\setminus \{0\}$ such that $\beta y^m\preceq \alpha$. Then, there exists $\delta \succeq 0$ such that $\alpha =\beta y^m+\delta$. Since $\beta y^m\succ 0$, we have $\alpha\succ \delta$, and thus $\N(\alpha)>\N(\delta)$. By definition, it follows that $\delta\in \operatorname{Rep}(Q)$. This completes the proof of (1).

    For (2), it is clear that the rank of $Q\perp\langle\beta\rangle$ is $n+1$ as we add an extra coefficient $\beta$. For the norm inequality; since $\beta \in E(Q,\alpha)$, by (1) there exists $y\in \co_K\setminus\{0\}$ such that $\beta y^m\preceq \alpha$. Thus, $\N(\alpha)\geq \N(\beta y^m)=\N(\beta)\N(y)^m\geq \N(\beta)$ as $m$ is even.
\end{proof}

Let us now show that for a given form $Q$, the sets $T(Q)$ and $E(Q)$ are both finite. Hence, the set of all escalations is finite. Note that the finiteness of both of these sets is clear if $Q$ is universal.

\begin{lemma}\label{finitenessof truant}
     Let $m\geq 2$ be an even integer, $K$ a totally real number field, and $Q$ a totally positive definite non-universal diagonal $m$-ic form over $K$. Then:
     \begin{enumerate}
         \item $T(Q)$ is finite.
         \item For each $\alpha \in T(Q)$, the set $E(Q,\alpha)\neq \emptyset$ is finite.
         \item $E(Q)$ is finite.
     \end{enumerate}
\end{lemma}

\begin{proof}
     \begin{enumerate}
         \item By definition, all truants have the same norm $\min \mathrm{N}(\operatorname{Unrep}(Q))$. Since the set of all elements in $\co_K^+/\co
         _K^{\times m}$ with norm $\leq \min \mathrm{N}(\operatorname{Unrep}(Q))$ is finite, the set $T(Q)$ is finite.

         \item Let us fix an $\alpha\in T(Q)$. It is clear that $E(Q,\alpha)\neq \emptyset$ as $\alpha$ always belongs to $E(Q,\alpha)$. By Lemma \ref{lem:3.6}, we have 
         \[
        \begin{aligned}
        E(Q,\alpha)
        &=\{\beta \in \co_K^+/\co_K^{\times m}
        \mid \beta y^m\preceq \alpha
        \text{ for some } y\in \co_K\setminus\{0\}\}\\
        &\subseteq
        \{\beta \in \co_K^+/\co_K^{\times m}
        \mid \N(\beta)\leq \min \mathrm{N}(\operatorname{Unrep}(Q))\}.
        \end{aligned}
        \]
         Again, the later set is finite, so $E(Q,\alpha)$ is also finite.

         \item Note that $$E(Q)=\bigcup_{\alpha\in T(Q)} E(Q,\alpha).$$ Since $T(Q)$ and $E(Q,\alpha)$ are both finite, it follows that $E(Q)$ is finite.
     \end{enumerate}
\end{proof}

\begin{definition}
    Let $Q=\langle a_1,a_2,\dots,a_n\rangle$ be a totally positive definite diagonal $m$-ic form over $K$, and $I\subseteq \{1,2,\dots,n\}$. The form $$Q_I=\langle a_i\mid i\in I\rangle$$ is called a subform of $Q$.
\end{definition}

Note that, by definition, $\operatorname{Rep}(Q_I)\subseteq \operatorname{Rep}(Q)$ for all $I\subseteq\{1,2,\dots,n\}$. In particular, if a subform $Q_I$ is universal, then the form $Q$ is also universal.

\begin{lemma}\label{lem:3.9}
    Let $m\geq 2$ be an even integer, $K$ a totally real number field, $Q=\langle a_1,a_2,\dots,a_n\rangle$ a totally positive definite diagonal $m$-ic form over $K$, $I\subsetneq \{1,2,\dots,n\}$, and $Q_I$ a subform of $Q$ that is not universal. Suppose $\alpha\in T(Q_I)$ is a truant and $Q$ represents $\alpha$. Then, there exists $j\in \{1,2,\dots,n\}\setminus I$ such that $a_j\in E(Q_I,\alpha)$.
\end{lemma}

\begin{proof}
    Since $Q$ represents $\alpha$, there exists $x\in \co_K^n$ such that $Q(x)=\alpha$, i.e., 
    \[
    \alpha=Q(x)=\sum_{i=1}^n a_ix_i^m=\sum_{i\in I} a_ix_i^m+\sum_{i\notin I} a_ix_i^m.
    \]
    Note that the element $\sum_{i\in I} a_ix_i^m\in \operatorname{Rep}(Q_I)$. As $\alpha\notin \operatorname{Rep}(Q_I)$, there exists $j\in \{1,2,\dots,n\}\setminus I$ such that $x_j\neq 0$ and $a_jx_j^m\preceq \alpha$. Therefore, $a_j\in E(Q_I,\alpha)$ thanks to Lemma \ref{lem:3.6}.
\end{proof}

We now create an algorithm that provides a constructive method of searching for universal diagonal $m$-ic forms. This algorithm indeed provides all \emph{proper} universal diagonal $m$-ic forms (i.e., none of its subform is universal).

\begin{algorithm}[H]
\caption{(Escalation of diagonal $m$-ic forms)}
\label{alg:escalation}
\begin{algorithmic}[1]
    \Input A totally real number field $K$ and an even integer $m\ge 2$.
    \Output A set $\mathcal{U}$ of totally positive definite diagonal $m$-ic forms over $K$ containing all proper universal forms.
    \State Initialize the set of forms to process: $\mathcal{Q} \gets \{Q_0\}$, where $Q_0$ is the empty form.
    \State Initialize the set of universal forms found: $\mathcal{U} \gets \emptyset$.
    \While{$\mathcal{Q} \neq \emptyset$}
        \State Choose and remove a form $Q$ from $\mathcal{Q}$.
        \State $T \gets T(Q)$ 
             \Comment{$T := \{\,\alpha \in (\co_K^+/\co_K^{\times m}) \setminus \operatorname{Rep}(Q) \mid \N(\alpha) = \min_{\beta \in (\co_K^+/\co_K^{\times m})\setminus \operatorname{Rep}(Q)} \operatorname{N}(\beta)\,\}$}
        \If{$T = \emptyset$}
            \State $\mathcal{U} \gets \mathcal{U} \cup \{Q\}$ \Comment{$Q$ represents all of $\co_K^+/\co_K^{\times m}$, hence is universal}
        \Else
            \ForAll{$\alpha \in T$}
                \ForAll{ $\beta\in E(Q,\alpha)$ }
                    \State $\mathcal{Q} \gets \mathcal{Q}\cup \{Q\perp \langle \beta\rangle\}$
                \EndFor    
            \EndFor
        \EndIf
    \EndWhile
    \State \Return $\mathcal{U}$
\end{algorithmic}
\end{algorithm}

\begin{rem}
\noindent
\begin{enumerate}
    \item While $T(Q)$ exists and is finite (Lemma \ref{finitenessof truant}), its computation is problematic. If the form $Q$ is universal, the algorithm should return $T(Q)=\emptyset$; but, as written, the algorithm does not recognize such a situation. Nevertheless, Algorithm \ref{alg:escalation} is an important theoretical tool, and one must check universality by other means. A prominent way to check for universality is to reduce the problem of checking all elements to just checking up to some bound. In our case, we will accomplish this using the asymptotic local--global principle (Corollary \ref{cor:main}), which we will prove later. 

    \item The for loop ($11$--$13$) can be made much more efficient by computing $E(Q,\alpha)$ a single $\alpha\in T$. By doing so, we can hope that many of the other truants are already represented by the single $\alpha$ escalation. However, as we mentioned earlier, this process might miss many forms, even though it will generate all critical elements.

    \item Let us also mention that both loops in the algorithm are finite, as $T(Q)$ and $E(Q,\alpha)$ are finite.
\end{enumerate}
\end{rem}

In the next proposition, we will prove that if Algorithm \ref{alg:escalation} terminates, then there exists a finite criterion set.

\begin{proposition}\label{lem:finitecriterionset}
Let $K$ be a totally real number field and $m\geq 2$ be an even integer. Assume that Algorithm \ref{alg:escalation} terminates after processing finitely many diagonal $m$-ic forms.
Let
\[
\mathcal{C}(K,m)=\bigcup_{Q \text{ processed by Algorithm \ref{alg:escalation}}} T(Q)
\]
be the union of all truant sets that appeared during the execution.
Then, $\mathcal{C}(K,m)$ is a finite criterion set for diagonal $m$-ic forms over $K$.
\end{proposition}

\begin{proof}
Since Algorithm \ref{alg:escalation} terminates after processing finitely many diagonal $m$-ic forms, it follows that $\mathcal{C}(K,m)$ is finite thanks to Lemma \ref{finitenessof truant}. It remains to show that $\mathcal{C}(K,m)$ is a criterion set, i.e., a totally positive definite diagonal $m$-ic form $Q$ is universal if and only if it represents all elements of $\mathcal{C}(K,m)$.

It is obvious that if $Q$ is universal, then it represents all elements of $\mathcal{C}(K,m)$. Conversely, assume that $Q=\langle a_1,\dots,a_n\rangle$ represents all elements of $\mathcal{C}(K,m)$. We want to show that $Q$ is universal. 

We construct 
\[
\emptyset=I_0\subsetneq I_1\subsetneq I_2\subsetneq \cdots \subsetneq \{1,2,\dots,n\}
\]
such that every subform $Q_{I_j}$ is a form processed by Algorithm \ref{alg:escalation}. We claim that for every $j\geq 0$, either $Q_{I_j}$ is universal, or $j<n$ and we can construct $I_{j+1}$ with $I_j\subsetneq I_{j+1}\subset \{1,2,\dots,n\}$. Assume that $Q_{I_j}$ is processed. If $Q_{I_j}$ is universal, then there is nothing to prove, as then $Q$ is universal. Therefore, assume $Q_{I_j}$ is non-universal and choose a truant $\alpha\in T(Q_{I_j})\subseteq \mathcal{C}(K,m)$. Note that $Q$ represents $\alpha$ because $Q$ represents all elements of $\mathcal{C}(K,m)$ and $\alpha\in \mathcal{C}(K,m)$. By Lemma \ref{lem:3.9}, there exists $k\in \{1,2,\dots,n\}\setminus I_j$ such that $$Q_{I_j \cup \{k\}}\in\{Q_{I_j}\perp \langle\beta\rangle\mid \beta \in E(Q_{I_j}, \alpha)\}.$$ As the algorithm inserts every form corresponding to elements of  $E(Q_{I_j}, \alpha)$ into the queue $\mathcal{Q}$, and the algorithm terminates, $Q_{I_j \cup \{k\}}$ is also a processed form. Now, set $I_{j+1}=I_j\cup \{k\}$. Then, we have $|I_{j+1}|=|I_j|+1$. This completes the proof of the claim. 

Since $|I_j|\leq n$, after at most $n$ steps, the construction must stop. However, from the claim, the construction stops only when $Q_{I_j}$ is universal. Hence, $Q$ is universal, as we wanted. 
\end{proof}

\begin{rem}
Under the termination assumption that we are yet to prove, we can now say that there is at least one finite criterion set, namely 
\[
\mathcal{C}_K^{(m)}=\mathcal{C}(K,m)=\bigcup_{Q \text{ processed by Algorithm \ref{alg:escalation}}} T(Q).
\]
In the next several sections, we prepare to prove the asymptotic local–-global principle, and finally, using that in Section \ref{S9}, we prove that the algorithm indeed terminates.
\end{rem}

By an \emph{escalation path}, we mean a sequence of diagonal $m$-ic forms $Q_0 \,( \text{empty form}), Q_1,Q_2,\dots$ with $Q_{i+1}=Q_i\perp\langle \beta_{i+1}\rangle$, where $\beta_{i+1}\in E(Q_i,\alpha_i)$ for some truant $\alpha_i\in T(Q_i)$. The rank of $Q_i$ is $i$, and the path is infinite if $T(Q_i)\neq \emptyset$ for all $i$. Running Algorithm \ref{alg:escalation} yields a rooted tree, denoted throughout by $\mathcal{T}$, with root $Q_0$.

In the following lemma, we record some obvious properties of forms that appear in an escalation path. 

\begin{lemma}\label{lem:3.8}
    Let $Q_0, Q_1,Q_2,\dots$ with $Q_{i+1}=Q_i\perp\langle \beta_{i+1}\rangle$, where $\beta_{i+1}\in E(Q_i,\alpha_i)$ for some truant $\alpha_i\in T(Q_i)$. Then:
    \begin{enumerate}
        \item $\operatorname{Rep}(Q_i)\cup \{\alpha_{i+1}\}\subseteq \operatorname{Rep}(Q_{i+1})$, and $\alpha_{i+1}\notin \operatorname{Rep}(Q_i)$.
        \item The classes $\alpha_1,\alpha_2,\dots$ are pairwise distinct.
        \item The sequence of minima $\min\mathrm{N}(\operatorname{Unrep}(Q_i))$ is non-decreasing in $i$. 
        \item If the path is infinite, then $\min\mathrm{N}(\operatorname{Unrep}(Q_i)))\rightarrow\infty$ as $i\rightarrow\infty$.
    \end{enumerate}
\end{lemma}

\begin{proof}
    \begin{enumerate}
        \item By definition, $Q_{i+1}=Q_i\perp\langle \beta_{i+1}\rangle$,  where $\beta_{i+1}\in E(Q_i,\alpha_i)$ for $\alpha_i\in T(Q_i)$, so $\operatorname{Rep}(Q_i)\subseteq \operatorname{Rep}(Q_{i+1})$. Also, by the definition of $E(Q_i,\alpha_i)$, the form $Q_{i+1}=Q_i\perp\langle \beta_{i+1}\rangle$ represents $\alpha_{i+1}$, so $\alpha_{i+1}\in \operatorname{Rep}(Q_{i+1})$. Since $\alpha_{i+1}\in T(Q_{i+1})$ and $Q_{i+1}=Q_i\perp\langle \beta_{i+1}\rangle$, it follows that $\alpha_{i+1}\notin \operatorname{Rep}(Q_i)$.

        \item Let $i<j$. Then, by (1), $\alpha_i\in \operatorname{Rep}(Q_{i})\subseteq \operatorname{Rep}(Q_{j-1})$, but $\alpha_j\notin \operatorname{Rep}(Q_{j-1})$. Therefore, $\alpha_i\neq \alpha_j$.

        \item By (1), we have $\operatorname{Unrep}(Q_{i+1})\subseteq \operatorname{Unrep}(Q_i)$, so the minimum satisfies $\min\mathrm{N}(\operatorname{Unrep}(Q_i)))\leq \min\mathrm{N}(\operatorname{Unrep}(Q_{i+1})))$ as required.

        \item Suppose there exists a $B>0$ such that $\min\mathrm{N}(\operatorname{Unrep}(Q_i)))\leq B$ for all $i$. Since $\N(\alpha_i)=\min\mathrm{N}(\operatorname{Unrep}(Q_i)))$, we have infinitely many distinct classes $\alpha_i$ with $\N(\alpha_i)\leq B$, which is impossible. Therefore, $\min\mathrm{N}(\operatorname{Unrep}(Q_i)))\rightarrow\infty$ as $i\rightarrow\infty$. \qedhere
    \end{enumerate}
\end{proof}

\section{Estimate from geometry of numbers} \label{S4}

As mentioned in the Introduction, the termination of the algorithm hinges primarily on an asymptotic local--global principle (Corollary \ref{cor:main}). To establish this principle, we first develop a connection between elements of sufficiently large norm and elements whose coordinates, with respect to a fixed integral basis, are sufficiently large.

This section is devoted to establishing this connection. More precisely, we prove that if $\alpha$ is a totally positive integer of sufficiently large norm, then there exists a unit $\varepsilon$ such that $\alpha\varepsilon^m$ has sufficiently large coordinates with respect to some integral basis. This is precisely the ingredient needed to relate the representation of elements of large norm by a diagonal $m$-ic form over $K$ to the representation of elements with sufficiently large coordinates by the corresponding system of integral $m$-ic forms obtained via the Weil restriction. Note that this phenomenon does not hold for all integral bases. We shall illustrate this with an example later. 

The following lemma can be viewed as a strengthening of \cite[Lemma $4.2$]{kp}, where only a lower bound is established. 

\begin{lemma}\label{lem:4.1}
    Let $m\geq 2$ be an even integer, and $K$ be a totally real number field of degree $d$. There exist constants $c_1',\, c_2'>0$, depending only on $m,\, K$ such that for every $\alpha\in\co_K^+$, there exists an $\varepsilon\in\co_K^{\times}$ such that the element $\beta:=\alpha\varepsilon^m$ satisfies 
    \[
    c_1'\N(\alpha)^{1/d}\leq \sigma_i(\beta)\leq c_2'\N(\alpha)^{1/d}
    \]
    for all embeddings $\sigma_i$.
\end{lemma}

\begin{proof}
    Let $\sigma_1,\dots,\sigma_d$ be real embeddings of $K$ into $\C$. Define the logarithmic map 
    \[
    L: \co_K^+\rightarrow \R^d: x\mapsto (\log \sigma_1(x),\dots, \log \sigma_d(x));
    \]
    and let 
    \[
    H=\left\{u\in \R^d\mid \sum_{i=1}^du_i=0\right\}
    \]
    be the trace zero hyperplane. For $x\in\co_K^+$, define $\xi(x)\in H$ by 
    \[
    L(x)=\frac{1}{d} \log \N(x)\cdot (1,1,\dots,1)+\xi(x).
    \]
    Observe that the map $\xi$ has the property that for $x,y\in\co_K^+$, $\xi(xy)=\xi(x)+\xi(y)$. Since $m$ is even, $\varepsilon^m\in \co_K^{\times,+}$ for all $\varepsilon\in\co_K^\times$. By the Dirichlet unit theorem, it follows that the image 
    \[
    \Lambda=\xi(\co_K^{\times m})\subset H
    \]
    is a full rank lattice, i.e., $\operatorname{rank}\Lambda=d-1=\dim H$. Consider the action of $\Lambda$ on $H$ by translation, i.e., for $\lambda\in\Lambda$ and $u\in H$, the action is $u\mapsto u+\lambda$; and choose a compact fundamental domain $\mathcal{F}\subset H$ for this action. For a given $1\leq i \leq d$, the coordinate projection 
    \[
    \pi_i: \mathcal F\rightarrow \R:(u_1,\dots,u_d)\mapsto u_i,
    \]
     is continuous, so by compactness, there exist $a,b\in\R$ such that 
    \begin{equation}\label{e:5.1}
        a\leq u_i\leq b,
    \end{equation}
    where $u=(u_1,\dots,u_d)\in \mathcal F$ and $1\leq i\leq d$. Set $c_1'=\exp(a)$ and $c_2'=\exp(b)$, then $0< c_1',c_2'<\infty$.

    Let $\alpha\in\co_K^+$ be arbitrary. By the definition of $\xi$, $\xi(\alpha)\in H$. Since $\mathcal F$ is a fundamental domain, there exists a unique $\lambda\in\Lambda$ such that $\xi(\alpha)+\lambda\in \mathcal{F}$. Thus, there exists an $\varepsilon\in\co_K^\times$ such that
    \[
    \xi(\beta)=\xi(\alpha\varepsilon^m)=\xi(\alpha)+\xi(\varepsilon^m)=\xi(\alpha)+\lambda\in \mathcal{F}.
    \]
    Since $\beta=\alpha\varepsilon^m$, we have $\N(\alpha)=\N(\beta)$. From the definition of $\xi$, it follows that 
    \[
    \sigma_i(\beta)=\N(\alpha)^{1/d}\exp(\xi(\beta)_i)
    \]
    for all $1\leq i\leq d$, where $\xi(\beta)_i$ is the $i$th coordinate of the vector $\xi(\beta)$. Applying the bounds from \eqref{e:5.1}, we obtain
    \begin{equation}\label{e:5.2}
        c_1'\N(\alpha)^{1/d}\leq\sigma_i(\beta)\leq c_2'\N(\alpha)^{1/d}
    \end{equation}
    for all $1\leq i\leq d$, as desired.
\end{proof}

Next, we prove that the codifferent of a given totally real number field has a $\Z$-basis consisting only of totally positive elements. 

\begin{lemma}\label{lem:4.2}
    Let $K$ be a totally real number field of degree $d$, and $\co_K^\vee$ be its codifferent. Then, there exists a $\Z$-basis $\{\omega_1',\omega_2',\dots,\omega_d'\}$ of $\co_K^\vee$ such that $\omega_i'\succ0$ for all $1\leq i\leq d$.
\end{lemma}

\begin{proof}
    Recall that the codifferent 
    \[
    \co^\vee_K=\{x\in K\mid \T(x\co_K)\subseteq \Z\}
    \]
    is a free $\Z$-module of rank $d$. The set $\T(\co_K)$ is a nonzero ideal in $\Z$; we write $\T(\co_K)=n\Z$ for some $n\in\Z_{>0}$. If $x\in \Q\cap\co^\vee_K$, then $\T(x\co_K)=x\T(\co_K)=xn\Z$, so $x\in \frac{1}{n}\Z$. Conversely, $\frac{1}{n}\in\co^\vee_K$ because $\frac{1}{n}\T(\co_K)=\Z$. Hence, $\Q\cap\co^\vee_K=\frac{1}{n}\Z$.

    The element $\frac{1}{n}$ is a totally positive and primitive element in the $\Z$-module $\co_K^\vee$; indeed, if $1/n=ky$ with $k\geq 2$ and $y\in\co_K^\vee$, then $y=\frac{1}{kn}\in \Q\cap\co^\vee_K=\frac{1}{n}\Z$, which forces $k=1$.  Note that any primitive element can be extended to a basis. Thus, there exist $\gamma_2,\dots,\gamma_d\in\co^\vee_K$ such that 
    \[
    \co^\vee_K=\Z\cdot\frac{1}{n}\oplus\Z\cdot \gamma_2\oplus \cdots\oplus \Z\cdot\gamma_d.
    \]
    Choose positive integers $m_2,m_3,\dots,m_d$ large enough that for every embedding $\sigma_i$, we have
    \[
    \sigma_i(\gamma_j)+\frac{m_j}{n}>0,
    \]
    for all $2\leq j\leq d$. Define 
    \[
    \omega_1'=\frac{1}{n}\quad \omega_j'=\gamma_j+m_j\omega_1'\quad \forall\, 2\leq j\leq d.
    \]

    Since 
    \[
    (\omega_1',\dots,\omega_d')=(1/n, \gamma_2,\dots,\gamma_d) \begin{pmatrix}
    1 & m_2 & m_3 & \cdots & m_d \\
    0 & 1 & 0 & \cdots & 0 \\
    0 & 0 & 1 & \cdots & 0 \\
    \vdots & \vdots & \vdots & \ddots & \vdots \\
    0 & 0 & 0 & \cdots & 1
    \end{pmatrix};
    \]
    the set $\{\omega_1',\omega_2',\dots,\omega_d'\}$ is a $\Z$-basis for $\co^\vee_K$, and by construction, every $\omega_i'$ is totally positive, as required.
\end{proof}

We now prove the main result of this section. 

\begin{proposition}\label{unitadjust}
     Let $m\geq 2$ be an even integer and $K$ be a totally real number field of degree $d$. Consider a $\Z$-basis of totally positive elements for $\co_K^\vee$  from Lemma \ref{lem:4.2}, and let $\{\omega_1,\omega_2,\dots,\omega_d\}$ be its dual basis with respect to the trace pairing. Then, the following statements hold:
     \begin{enumerate}
         \item The dual basis $\{\omega_1,\omega_2,\dots,\omega_d\}$ is an integral basis for $\co_K$.
         \item There exist constants $c_1,c_2>0$ that depend only on $m,\, K$, and $\omega_1,\dots,\omega_d$, such that for every $\alpha\in\co_K^+$, there exists an $\varepsilon\in\co_K^\times$ with the property that $\beta:=\varepsilon^m\alpha=\sum_{i=1}^d \beta_i\omega_i$, $\beta_i\in\Z$ satisfies 
        \[
        c_1\N(\alpha)^{1/d}\leq \beta_i\leq c_2\N(\alpha)^{1/d}
        \]
        for all $1\leq i\leq d$. 
     \end{enumerate}
\end{proposition}

\begin{proof}
     Consider a $\Z$-basis of totally positive elements $\{\omega_1',\omega_2',\dots,\omega_d'\}$ for $\co_K^\vee$, and let $\{\omega_1,\omega_2,\dots,\omega_d\}$ be its dual basis with respect to the bilinear trace-pairing $\langle x,y\rangle=\T(xy)$ given by 
    \[
    \T(\omega_i'\omega_j)=\delta_{ij}=\begin{cases}
        1, \text{ if } i=j\\
        0, \text{ if } i\neq j.
    \end{cases}
    \]
     Then, it follows that ${\omega_1,\omega_2,\dots,\omega_d}$ is a $\mathbb{Z}$-basis of the dual module $(\co_K^\vee)^\ast=\co_K$, i.e., ${\omega_1,\omega_2,\dots,\omega_d}$ is an integral basis of $\co_K$, thereby proving (1).

     Now, we write $\beta=\sum_{i=1}^d \beta_i\omega_i$ with $\beta_i\in\Z$. By duality, we have 
    \begin{equation}\label{eq:4.3}
        \beta_i=\T(\beta\omega_i')=\sum_{j=1}^d\sigma_j(\beta)\sigma_j(\omega_i').
    \end{equation}
    Note that both $\beta$ and $\omega_i'$ are totally positive, so every term in the sum \eqref{eq:4.3} is positive. Since $\omega_i'\in\co^\vee_K$ and $1\in\co_K$, we have that 
    \[
    \T(\omega_i')=\sum_{j=1}^d\sigma_j(\omega_i')\in \Z,
    \]
    moreover, by totally positivity $\T(\omega_i')\geq 1$. By Lemma \ref{lem:4.1} there exist constants $c_1',\, c_2'>0$ such that
    \[
    c_1'\N(\alpha)^{1/d}\sigma_j(\omega_i')\leq \sigma_j(\beta)\sigma_j(\omega_i')\leq c_2'\N(\alpha)^{1/d}\sigma_j(\omega_i').
    \]
    Summing over $j=1,2,\dots,d$ yields
    \[
    c_1'\N(\alpha)^{1/d}\T(\omega_i')\leq \beta_i\leq c_2'\N(\alpha)^{1/d}\T(\omega_i').
    \]
    Finally, set 
    \[
    c_1=c_1'\min_{1\leq i\leq d} \T(\omega_i'),\quad c_2=c_2'\max_{1\leq i\leq d} \T(\omega_i');
    \]
    these constants are positive and depend only on $m,\,K$, and the basis $\{\omega_1,\omega_2,\dots,\omega_d\}$, we obtain
    \[
    c_1\N(\alpha)^{1/d}\leq \beta_i\leq c_2\N(\alpha)^{1/d},
    \]
    for all $1\leq i\leq d$. This proves (2).
\end{proof}

The following example illustrates that Proposition \ref{unitadjust}(2) does not hold for every integral basis.

\begin{eg}
    Let $K=\Q(\sqrt 2)$ and $m=2$. Choose the integral basis $1,\, 2+\sqrt2$, and take $\alpha=1$. Suppose there exist constants $c_1,c_2>0$ such that $\beta=\varepsilon^2=a+b(2+\sqrt{2})$, $a,b\in\Z$ satisfy $c_1\leq a,b\leq c_2$. In particular, $a,b>0$. 

    Note that the units in $\Z[\sqrt 2]$ are $\pm(1+\sqrt2)^n$, where $n\in \Z$. Therefore, $\varepsilon^2=(1+\sqrt2)^{2n}$. If $n=0$, then $\varepsilon^2=1$, so $b=0$, a contradiction. If $n>0$, write $(1+\sqrt 2)^{2n}=x+y\sqrt 2$, $x,y\in\Z_{>0}$. Then, $x^2-2y^2=1$, so for $y\geq 1$, we have 
    \[
    x^2=2y^2+1<4y^2,
    \]
    hence $x<2y$. Therefore, $x+y\sqrt 2=(x-2y)+y(2+\sqrt2)$, so $a=x-2y<0$, a contradiction. Finally, if $n<0$, say $n=-r$ with $r>0$, then $(1+\sqrt 2)^{-2r}=(3-2\sqrt 2)^r=x-y\sqrt{2}$ with $x,y>0$. Thus, 
    \[
    x-y\sqrt{2}=(x+2y)-y(2+\sqrt 2),
    \]
    so $b=-y<0$, again a contradiction.
\end{eg}

\section{Preparation for asymptotic local--global principle}

In this section, we recall Birch's theorem, which we use to establish the asymptotic local--global principle and develop the necessary technical machinery for its application in our setting.

From this point onward, we will work with the fixed integral basis $\omega_1,\omega_2,\dots,\omega_d$ of $\co_K$ obtained in Proposition \ref{unitadjust}. We begin by establishing the equivalence between the representation of an element by a totally positive definite diagonal $m$-ic form and the representation of its coordinates, with respect to this integral basis, by a system of integral $m$-ic forms.

\begin{lemma}\label{lem:5.1}
     Let $m\geq 2$ be even integer, $K$ be a totally real number field of degree $d$, and $Q$ be a totally positive definite diagonal $m$-ic form in $n$ variables over $K$. There exist $m$-ic forms $F_1,F_2,\dots,F_d$ in $nd$ variables $x=(x_{1,1}, x_{1,2},\dots,x_{n,d})$ over $\Z$ such that 
     \[
    Q(x)=Q(x_1,x_2,\dots,x_n)=\sum_{k=1}^d F_k(x)\omega_k,
     \]
     where $\omega_1,\dots,\omega_d$ is an integral basis from Proposition \ref{unitadjust}.
\end{lemma}

\begin{proof}
    Consider the totally positive definite diagonal $m$-ic form over $K$
    \[
    Q(x)=Q(x_1,\dots,x_n)=\sum_{i=1}^{n} a_i x_i^{m},
    \]
     i.e., $a_1,\dots,a_n\in \mathcal{O}_K^+$. We write 
    \[
    x_i=\sum_{j=1}^d x_{i,j} \omega_j, \quad a_i=\sum_{t=1}^d a_{i,t} \omega_t,
    \]
    where $x_{i,j},\, \alpha_k,\,a_{i,t}\in \Z$ for all $i,j,\text{ and }t$. Then,  $Q(x)$ becomes 
    \[
    Q(x)=\sum_{i=1}^n a_i \left(\sum_{j=1}^d x_{i,j}\omega_j \right)^m.
    \]
    We multiply everything out using the identity 
    \[
    \omega_i\omega_j=\sum_{k=1}^d c_{i,j}^{k} \omega_k, \quad c_{i,j}^{k}\in\Z.
    \]
    For each $i$, the power $(\sum_j x_{i,j}\omega_j)^m$ is a sum of monomials of degree $m$ in variables $x_{i,1},\dots,x_{i,d}$ multiplied by products of $m$ basis elements. Thus, we have 
    \[
    \left(\sum_{j=1}^d x_{i,j}\omega_j\right)^m=\sum_{k=1}^d P_{i}^k(x_{i,1},\dots,x_{i,d})\omega_k,
    \]
    where each $P_i^k$ is an $m$-ic form with integer coefficients. Next, multiply by $a_i=\sum_t a_{i,t}\omega_t$ to obtain 
    \[
    a_ix_i^m=\left(\sum_{t=1}^d a_{i,t}\omega_t \right)\left(\sum_{r=1}^d P_i^r\omega_r \right)=\sum_{k=1}^d\left(\sum_{t,r}a_{i,t}c_{t,r}^k P_i^r(x_{i,1},\dots,x_{i,d}) \right)\omega_k.
    \]
    
    Define
    \[
    F_i^k(x_{i,1},\dots,x_{i,d}):=\sum_{t,r} a_{i,t}c_{t,r}^k P_i^r(x_{i,1},\dots,x_{i,d}).
    \]
    Note that each $F_i^k$ is an $m$-ic form with integer coefficients. Now, summing over $i$, the coefficient of $\omega_k$ in $Q(x)$ is 
    \begin{equation}\label{fk}
    F_k({x}):=\sum_{i=1}^n F_i^k(x_{i,1},\dots,x_{i,d}), \quad \forall\, 1\leq k\leq d,
    \end{equation}
    where $x=(x_{1,1},\dots, x_{1,d},x_{2,1}, \dots, x_{n,d})\in\Z^{nd}$. Each $F_k$ is an $m$-ic form in $nd$ variables with coefficients in $\Z$. Therefore, we get 
    \[
    Q(x)=Q(x_1,x_2,\dots,x_n)=\sum_{k=1}^d F_k(x)\omega_k,
    \]
    as we wanted.
\end{proof}

\begin{lemma}\label{lem:equivalentsystem}
    Let $m\geq 2$ be even integer, $K$ be a totally real number field of degree $d$, $\alpha\in\co_K^+$, $Q$ be a totally positive definite diagonal $m$-ic form in $n$ variables over $K$, and $F_1,\dots,F_d$ are the $m$-ic forms from Lemma \ref{lem:5.1}. Write $\alpha=\sum_{k=1}^d \alpha_k\omega_k$ with $\alpha_k\in\Z$, where $\omega_1,\dots,\omega_d$ is an integral basis from Proposition \ref{unitadjust}. Then, the following statements are equivalent:
    \begin{enumerate}
        \item $\alpha$ is represented by $Q$;
        \item there exists $x\in \Z^{nd}$ such that $F_k(x)=\alpha_k$ for all $1\leq k\leq d$.
    \end{enumerate}
\end{lemma}

\begin{proof}
    Write $\alpha=\sum_{k=1}^d \alpha_k \omega_k$. By Lemma \ref{lem:5.1}, 
    \[
    Q(x)=\sum_{k=1}^d F_k(x)\omega_k.
    \]
    Since $\omega_1,\dots,\omega_d$ is an integral basis for $\co_K$, the equivalence of the statements follows immediately.
\end{proof}

 For a fixed $\mu=(\mu_1,\mu_2,\dots,\mu_d)\in\C^d$, let
\begin{equation}\label{eq:VF}
V_F(\mu)=\left\{ x\in \C^{nd} :
F_k(x)=\mu_k \,\forall\, 1\leq k\leq d \right\}
\end{equation}
be the (affine) variety defined by $F=(F_1,\ldots,F_d)$, where $F_k$ are as in Lemma \ref{lem:5.1}. We write $V_F(\alpha)$ to denote the variety corresponding to $\mu=(\alpha_1,\dots,\alpha_d)$, where $\alpha_k$ are the coordinates of $\alpha$ when expressed in terms of the integral basis. The \emph{singular locus} of $V_F(\alpha)$ is defined by 
\[
V_F^*(\mu)=\left\{x\in \mathbb C^{nd} : \operatorname{rank} J_F(x)<d
\right\},
\]
where
\[
J_F(x)=\left(\frac{\partial F_k}{\partial x_{i,j}}(x)\right)_{\substack{1\le k,j\le d\\1\le i\le n}}
\]
is the $d\times nd$ Jacobian matrix of the variety $V_F(\mu)$. Let 
\[
V^\ast=\bigcup_{\mu\in\C^d}V_F^\ast(\mu)
\]
be the union of the loci of singularities of $V_F(\mu)$. 

We now restate Birch's theorem in the form that will be used throughout this work. For the reader's convenience, we also include a brief proof, which primarily consists of assembling several results from Birch's original paper and making explicit the facts that will be used later.

\medskip 

Let $F=(F_1,F_2,\dots,F_d)\in\Z[x_1,x_2,\dots,x_{nd}]^d$ be a tuple of $m$-ic forms in $nd$ variables. For $\nu\in\Z^d$, $q\in\Z_{\geq 1}$, and $a=(a_1,\dots,a_d)\in\Z^d$ with $\gcd(a_1,\dots,a_d,q)=1$, define
\[
S_{a,q}(\nu)= e\left(-\frac{a\cdot \nu}{q}\right)\sum_{x \mod q} e\left(\frac{a\cdot F(x)}{q} \right)\\=\sum_{x\mod q} e\left(\frac{1}{q}\sum_{k=1}^da_k(F_k(x)-\nu_k)\right),
\]
where $e(z)=\exp(2\pi iz)$. Furthermore, for a prime $p$, define 
\[
\mathfrak{S}_p(\nu)=\sum_{r=0}^\infty p^{-rnd}\sum_{\substack{{a \mod p^r}\\{\gcd(a_1,\dots, a_d,p)=1}}}S_{a,p^r}(\nu),
\]
and 
\[
\mathfrak{S}(\nu)=\prod_{p}\mathfrak{S}_p(\nu),
\]
where the product runs over all primes.

\begin{theorem}[{\cite[Theorem $1$]{bir}}]\label{birchthm}
    Let $F=(F_1,F_2,\dots,F_d)\in\Z[x_1,x_2,\dots,x_{nd}]^d$ be a tuple of $m$-ic forms in $nd$ variables, and $\nu=(\nu_1,\nu_2,\dots,\nu_d)\in \Z^d$. Define the variety 
    \[
    V_F(\nu)=\{x\in\C^{nd}\mid F_k(x)=\nu_k \, \forall\, 1\leq k\leq d\},
    \]
    and let $V^\ast$ be the union of loci of singularities the variety. Define $K_0$ by $nd-\dim(V^\ast)=2^{m-1}K_0$ and suppose $K_0>d(d+1)(m-1)$. Let $\mathcal{B}\subset [-1,1]^{nd}$ be a closed box of side length at most $1$, and for $P\geq 2$, $M(P;\nu)$ be the number of integer points on $V(\nu)$ in the box $P\mathcal{B}$. Then, there exists $\delta>0$ such that, writing $\nu=P^m\mu$, 
    \[
    M(P;\nu)=P^{d(n-m)}\cdot \mathfrak{S}(\nu)\cdot \Psi(\mu)+O(P^{d(n-m)-\delta}),
    \]
    where $\mathfrak{S}(\nu)$ is the \emph{singular series} defined above, $\Psi$ is the \emph{singular integral} defined in \cite[Section $6$]{bir}, and the implied constant depends only on $F,\mathcal{B}$, and $\delta$; in particular it is independent of $P$ and $\nu$. Moreover, we have the following:

    \begin{enumerate}
        \item For every compact set $\mathcal C\subset \operatorname{int} \mathcal{B}\setminus V^\ast$, there exists a constant $c(\mathcal{C})>0$ such that $\Psi(F(x))\geq c$ for all $x\in\mathcal C$.

        \item There exists a prime $p_0$ such that 
        \[
        \prod_{p>p_0}\mathfrak{S}_p(\nu)\geq \frac{1}{2},
        \]
        uniformly for all $\nu\in\Z^d$.

        \item For every prime $p$, every integer $N\geq 0$, and every $\nu\in\Z^d$, we have 
        \[
        \sum_{r=0}^N p^{-rnd}\sum_{\substack{{a \mod p^r}\\{\gcd(a_1,\dots, a_d,p)=1}}}S_{a,p^r}(\nu)=p^{-Nd(n-1)} \mathcal{A}_p(N;\nu),
        \]
        where 
        \[
        \mathcal{A}_p(N;\nu)=\#\{x\in (\Z/p^N\Z)^{nd}\mid F_k(x)\equiv \nu_k\pmod {p^N} \, \forall 1\leq k\leq d\}.
        \]
        In particular, 
        \[
        \mathfrak{S}_p(\nu)=\lim_{N\rightarrow\infty}p^{-Nd(n-1)} \mathcal{A}_p(N;\nu)
        \]
    \end{enumerate}
\end{theorem}

\begin{proof}
    The asymptotic formula and (1) is as written in \cite[Theorem $1$]{bir}. For (2); recall by \cite[lemma $5.4$]{bir}, for all $\varepsilon>0$, for all $q\geq 1$, and for all $a\in\Z^d$ with $\gcd(a_1,\dots,a_d,p)=1$, we have 
    \[
    |S_{a,q}(\nu)|=\Bigg| \sum_{x \mod q} e\left(\frac{a\cdot F(x)}{q} \right) \Bigg|\ll_\varepsilon q^{nd-\frac{K_0}{d(m-1)}+\varepsilon}.
    \]

    Write 
    \[
    \mathfrak{S}_p(\nu)=1+\sum_{r=1}^\infty p^{-rnd}\sum_{\substack{{a \mod p^r}\\{\gcd(a_1,\dots, a_d,p)=1}}}S_{a,p^r}(\nu).
    \]

    For a fixed $r\geq 1$, we have 
   \[
    \begin{aligned}
    \Bigg|\sum_{\substack{a \bmod p^r\\ \gcd(a_1,\dots,a_d,p)=1}} S_{a,p^r}(\nu)\Bigg|
    &\leq \#\{a\in(\mathbb Z/p^r\mathbb Z)^d \mid \gcd(a_1,\dots,a_d,p)=1\}\cdot \max_{\substack{a \bmod p^r\\ \gcd(a_1,\dots,a_d,p)=1}}
    |S_{a,p^r}(\nu)| \\
    &\ll_\varepsilon p^{rd}\,
    p^{r\left(nd-\frac{K_0}{d(m-1)}+\varepsilon\right)} = p^{r\left(nd+d-\frac{K_0}{d(m-1)}+\varepsilon\right)}.
    \end{aligned}
    \]

    Therefore, we obtained 
    \[
    p^{-rnd}\Bigg|\sum_{\substack{a \bmod p^r\\ \gcd(a_1,\dots,a_d,p)=1}} S_{a,p^r}(\nu)\Bigg|\ll_\varepsilon p^{r\left(d-\frac{K_0}{d(m-1)}+\varepsilon\right)}.
    \]
    By assumption $K_0>d(d+1)(m-1)$, it follows that $K_0/d(m-1)>d+1$. Thus, there exists a $\theta>0$ such that 
    \[
    d-\frac{K_0}{d(m-1)}<-(1+2\theta). 
    \]
    Now, choose $\theta=\varepsilon$, we have
    \[
    p^{-rnd}\Bigg|\sum_{\substack{a \bmod p^r\\ \gcd(a_1,\dots,a_d,p)=1}} S_{a,p^r}(\nu)\Bigg|\ll_\varepsilon p^{-r(1+\varepsilon)},
    \]
    as $-(1+2\theta)+\varepsilon=-(1+\varepsilon)$.
    Therefore, the tail of $\mathfrak{S}_p(\nu)$
    \begin{align*}
    |\mathfrak{S}_p(\nu)-1|=\Bigg| \sum_{r=1}^\infty p^{-rnd}\sum_{\substack{{a \mod p^r}\\{\gcd(a_1,\dots, a_d,p)=1}}}S_{a,p^r}(\nu)\Bigg|\ll_\varepsilon p^{-(1+\varepsilon)},
    \end{align*}
    uniformly for all $\nu$. Then, there exists a $p_0$ such that 
    \[
    \prod_{p>p_0} \mathfrak{S}_p(\nu)\geq \frac{1}{2}.
    \]
    Note that here we used the fact that $\mathfrak{S}_p(\nu)\geq 0$, which follows from (3). Let us now prove (3). For all $y\in (\Z/p^N\Z)^d$, by the orthogonality of additive characters, it follows that 
    \[
    p^{-Nd}\sum_{a\in(\Z/p^N\Z)^d} e\left(\frac{a\cdot y}{p^N} \right)=\begin{cases}
        1, \quad y\equiv 0\pmod {p^N}\\
        0 \quad \text{otherwise}.
    \end{cases}
    \]
    Applying this with $y=F(x)-\nu=(F_1(x)-\nu_1,\dots,F_d(x)-\nu_d)$, and summing over $x\in(\Z/p^N\Z)^{nd}$, we obtain 
    \[
    \mathcal{A}_p(N;\nu)=\sum_{x \mod p^N}p^{-Nd}\sum_{a\mod p^N} e\left( \frac{a\cdot (F(x)-\nu)}{p^N}\right).
    \]
    Interchanging the sums and separating the exponential yields 
    \begin{equation}\label{e:5.3}
        \mathcal{A}_p(N;\nu)= p^{-Nd}\sum_{a\mod p^N} e\left(- \frac{a\cdot \nu}{p^N}\right)\sum_{x\mod p^N} e\left( \frac{a\cdot F(x)}{p^N}\right).
    \end{equation}

    Note that every $a\in(\Z/p^N\Z)^d$ can be uniquely written as $a=p^{N-r}b$, where $0\leq r\leq N$ and $b\in (\Z/p^r\Z)^d$ with $\gcd(b_1\dots,b_d,p)=1$. Thus, we have 
    \[
    e\left( \frac{a\cdot F(x)}{p^N}\right)=e\left( \frac{b\cdot F(x)}{p^r}\right).
    \]
    Write $x=y+p^rz$ with $y\in (\Z/p^r\Z)^{nd}$ and $z\in (\Z/p^{N-r}\Z)^{nd}$. Since $F(x)$ is integral, by Taylor expansion, we obtain 
    \[
    F(x)\equiv F(y+p^rz)\equiv F(y) \pmod{p^r},
    \]
    so the character $e\left( \frac{b\cdot F(x)}{p^r}\right)$ depends only on the residue class modulo $p^r$. Summing over the $p^{(N-r)nd}$ choices for $z$, we get 
    \begin{equation}\label{e:5.4}
        \sum_{x \mod p^N}e\left( \frac{a\cdot F(x)}{p^N}\right)=p^{(N-r)nd} \sum_{y \mod p^r} e\left( \frac{b\cdot F(y)}{p^r}\right).
    \end{equation}

    By the same argument applied to $e\left( -\frac{a\cdot \nu}{p^N}\right)$, we obtain 
    \begin{equation}\label{e:5.5}
        e\left( -\frac{a\cdot \nu}{p^N}\right)=e\left( -\frac{b\cdot \nu}{p^r}\right)
    \end{equation}
    Substituting the decomposition of $a$, and using \eqref{e:5.4},\eqref{e:5.5}, we obtained
    \[
    \begin{aligned}
    \mathcal{A}_p(N;\nu)
    &= p^{-Nd}
       \sum_{r=0}^N
       \sum_{\substack{b \bmod p^r\\ \gcd(b_1,\dots,b_d,p)=1}}
       p^{(N-r)nd}
       e\left(-\frac{b\cdot\nu}{p^r}\right)
       \sum_{y \bmod p^r}
       e\left(\frac{b\cdot F(y)}{p^r}\right) \\
    &= p^{-Nd}
       \sum_{r=0}^N
       p^{(N-r)nd}
       \sum_{\substack{b \bmod p^r\\ \gcd(b_1,\dots,b_d,p)=1}}
       S_{b,p^r}(\nu) \\
    &= p^{Nd(n-1)}
       \sum_{r=0}^N
       p^{-rnd}
       \sum_{\substack{b \bmod p^r\\ \gcd(b_1,\dots,b_d,p)=1}}
       S_{b,p^r}(\nu).
    \end{aligned}
    \]
    Consequently, 
    \[
    \mathfrak{S}_p(\nu)=\lim_{N\rightarrow \infty} \sum_{r=0}^N p^{-rnd}\sum_{\substack{{b \mod p^r}\\{\gcd(b_1,\dots, b_d,p)=1}}}S_{b,p^r}(\nu)=\lim_{N\rightarrow\infty}p^{-Nd(n-1)} \mathcal{A}_p(N;\nu).
    \]
    This completes the proof of (3).
\end{proof}

 We now compute the dimension of the union of loci of singularities of our variety \eqref{eq:VF}, which is necessary for the application of Birch's theorem.

\begin{proposition}\label{singdim}
    Let $m\geq 2$ be an even integer, $K$ be a totally real number field, and $Q$ be a totally positive definite diagonal $m$-ic form in $n$ variables over $K$. For $\mu\in\C^d$, let $V_F(\mu)$ be the variety defined in \eqref{eq:VF}, and $V^\ast$ be the union of loci of singularities of $V_F(\mu)$.
    Then, $\dim V^\ast=nd-n$.    
\end{proposition} 

\begin{proof}
    Recall from Lemma \ref{lem:5.1}, the $m$-ic forms $F_1,\dots,F_d\in \Z[x]$, where $x=(x_{1,1},\dots,x_{n,d})$, are defined by the expression 
    \begin{equation}\label{e:4.1}
        \sum_{i=1}^n a_i\left(\sum_{j=1}^d x_{i,j}\omega_j\right)^m=\sum_{k=1}^d F_k(x)\omega_k.
    \end{equation}
    Let $\sigma_1,\dots\sigma_d: K\hookrightarrow\C$ be the distinct embeddings of $K$ into $\C$. Apply $\sigma_r$ to both sides of \eqref{e:4.1}. Since the coefficients $a_i$ and the basis $\omega_j$ lie in $K$, we obtain
    \begin{equation}\label{e:4.2}
        \sum_{i=1}^n \sigma_r(a_i)\left(\sum_{j=1}^d x_{i,j}\sigma_r(\omega_j)\right)^m=\sum_{k=1}^d F_k(x)\sigma_r(\omega_k).
    \end{equation}
    for all $1\leq r\leq d$. For each $r$ and each $i$, we define 
    \[
    y_{i,r}(x)=\sum_{j=1}^d x_{i,j}\sigma
    _r(\omega_j).
    \]
    Since the matrix $\Delta=(\sigma_r(\omega_j))_{1\leq j,r\leq d}$ is invertible, the collection 
    \[
    \{y_{i,r}\mid 1\leq i\leq n, \, 1\leq r\leq d\}
    \]
    forms a set of $nd$ independent linear coordinates on $\C^{nd}$. In these new coordinates, \eqref{e:4.2} becomes 
    \[
    \sum_{i=1}^n \sigma_r(a_i)y_{i,r}^m=\sum_{k=1}^d F_k(x)\sigma_r(\omega_k).
    \]
    So, the original system $F_k(x)=\mu_k$ is equivalent to 
    \begin{equation}\label{e:4.3}
        \sum_{i=1}^n\sigma_r(a_i)y_{i,r}^m=\nu_r:=\sum_{k=1}^d \mu_k\sigma_r(\omega_k) \quad \forall\, 1\leq r\leq d.
    \end{equation}

    By definition, a point $x$ is singular on $V_F(\mu)$ precisely when the Jacobian $(\partial F_k/\partial x_{i,j})_{i,j}$ has rank $<d$. Since the change of variables $x_{i,j}\mapsto y_{i,j}$ is a linear isomorphism over $\C$ (this is because $y_{i,j}$ is obtained by multiplication by the matrix $\Delta$, which is invertible), the rank of the original Jacobian is equal to the rank of the Jacobian of the new system \eqref{e:4.3}.

    For a fixed $r$, $F_r$ depends only on $y_{1,r},\dots, y_{n,r}$, and 
    \[
    \frac{\partial F_r}{\partial y_{i,s}}=\begin{cases}
        m \sigma_r(a_i)y_{i,r}^{m-1}\quad \text{if }s=r,\\
        0 \quad \text{if }s\neq r.
    \end{cases}
    \]
    Therefore, the Jacobian is a block diagonal matrix whose $r$th block is the $1\times n $ row 
    \[
    (m \sigma_r(a_1)y_{1,r}^{m-1},\dots, m \sigma_r(a_n)y_{n,r}^{m-1}).
    \]
    Note that the rank of the full Jacobian is the number of nonzero rows. Since $a_i$ is totally positive, $\sigma_r(a_i)>0$ for all $1\leq r\leq d$; together with the fact that $m>0$, we have that the $r$th row is zero exactly when $y_{i,r}=0$ for all $1\leq i \leq n$. Observe that this condition is independent of the constants $\nu_r$, and hence of $\mu$. Consequently,
    \[
    V^\ast=\bigcup_{r=1}^d L_r,
    \]
    where $L_r=\{\mathbf y\in \C^{nd}\mid y_{i,r}=0 \, \forall\, 1\leq i\leq n\}$. Each $L_r$ is defined by $n$ independent linear equations, so it is a linear subspace of codimension $n$; so $\dim L_r=nd-n$. The union of the $d$ subspaces has the same dimension, i.e., $\dim V^\ast=nd-n$, as we wanted.  
\end{proof}

\section{The singular series}\label{S6}

In order to utilize Theorem \ref{birchthm} in our case, we need a nonzero main term, which is possible only if the singular series and the singular integral are uniformly bounded away from zero. In this section, we obtain a sufficient condition for the singular series to be uniformly bounded away from zero. 

Note that, by \cite[Theorem $1$]{bir}, $\mathfrak{S}(\nu)$ exceeds a positive constant uniformly in $\nu$ if the variety $V_F(\nu)$ has a non-singular mod $p$ point for all prime $p$. However, for $p|m$, this condition is almost impossible to satisfy for generic $m$-ic forms; in particular to our setup. Thus, we need a more refined sufficient condition for the positivity of singular series. 

While finding a solution mod $p$ is typically not the main difficulty in Birch's sufficient condition, ensuring that the solution is non-singular is considerably more delicate. The key idea of our approach is to replace this qualitative requirement with a quantitative one by measuring, in the sense of $p$-adic valuations, how close a solution mod $p$ is to being non-singular. 

\begin{lemma}\label{lem:6.1}
    Let $p$ be a prime, $A\in M_{d\times nd}(\Z_p)$ be a matrix. and consider the $\Z_p$-lattice $L=A\Z_p^{nd}$. Assume that  the index $[\Z_p^d:L]=p^T<\infty$ for some $T\geq 0$. Then, the ideal generated by all $d\times d$ minors of $A$, denoted by $I_d(A)$, is equal to $p^T\Z_p$. In particular, there exists a $d\times d$ minor of $A$ whose $p$-adic valuation is $T$.
\end{lemma} 

\begin{proof}
    Since $\Z_p$ is a principal ideal domain, by Smith normal form, there exist matrices $U\in GL_d(\Z_p)$, $V\in GL_{nd}(\Z_p)$ such that 
    \[
    UAV=D=\left(
    \begin{array}{cccc|ccc}
    p^{e_1} & 0 & \cdots & 0 & 0 & \cdots & 0 \\
    0 & p^{e_2} & \cdots & 0 & 0 & \cdots & 0 \\
    \vdots & \vdots & \ddots & \vdots & \vdots & & \vdots \\
    0 & 0 & \cdots & p^{e_d} & 0 & \cdots & 0
    \end{array}
    \right),
    \]
    where $0\leq e_1\leq \dots\leq e_d\leq \infty$ are the elementary divisors of $A$. Since $V$ is invertible, $\Z_p^{nd}=V\Z_p^{nd}$; thus $L=AV\Z_p^{nd}$, so $UL=UAV\Z_p^{nd}=D\Z_p^{nd}$. Note that 
    \[
    D\Z_p^{nd}=\bigoplus_{j=1}^d p^{e_j}\Z_p.
    \]
    Therefore, 
    \[
    L=U^{-1}D\Z_p^{nd} =U^{-1}\left( \bigoplus_{j=1}^d p^{e_j}\Z_p\right),
    \]
    and $[\Z_p^{d}:L]=p^{e_1+\dots+e_d}$. By our assumption on the index, it follows that $e_j<\infty$ for all $j$, and ${e_1+\dots+e_d}=T$. We will now show that $I_d(A)=p^T\Z_p$.

    If $S\subset \{1,2,\dots,nd\}$ is a set of $d$ column indices, then the columns of $UA$ indexed by $S$ are $(UA)_S=UA_S$, where $A_S$ is the $d\times d$ submatrix of $A$ with columns indexed by $S$. Therefore, $\det (UA)_S=\det U\det A_S$. Since $\det U$ is a unit, summing over all choices of $S$ gives $I_d(UA)=I_d(A)$. 

    Now, for the product $AV$, we use the \emph{Cauchy--Binet formula}. Again, for a choice of $S$ as above, $(AV)_S=AV_S$, where $V_S$ is a $nd\times d$ matrix formed by columns of $V$ indexed by $S$. Then, by the Cauchy--Binet formula, we have 
    \[
    \det (AV)_S=\sum_{\substack{S'\subset \{1,\dots,nd\}\\|S'|=d}} \det A_{S'}\det (V_S)_{S'},
    \]
    where $(V_S)_{S'}$ is the $d\times d$ submatrix of $V_S$ with rows indexed by $S'$. Thus, each $d\times d$ minor of $AV$ is a $\Z_p$-linear combination of the $d\times d$ minors of $A$, giving $I_d(AV)\subseteq I_d(A)$.

    Applying the same argument with $V^{-1}$ in place of $V$ gives 
    \[
    I_d(A)=I_d((AV)V^{-1})\subseteq I_d(AV).
    \]
    In total, we get $I_d(A)=I_d(AV)$. Combining both cases, we obtain
    \[
    I_d(A)=I_d(UA)=I_d(UAV)=I_d(D)=p^T\Z_p.
    \]

    Since $\Z_p$ is a discrete valuation ring, the ideal generated by the minors is principal, and the exponent is the minimum of the $p$-adic valuations of $d\times d$ minors. Hence, there exists a $d\times d$ minor with $p$-adic valuation exactly $T$, completing the proof.
\end{proof}

The following lemma already establishes connection between bounded $p$-adic valuations and the singular series. 

\begin{lemma}\label{lem:6.2}
    Let the notations and hypothesis be as in Theorem \ref{birchthm}. Let $\nu\in\Z^d$ and $p$ be a prime. Suppose there exists $\tilde{x}\in \Z_p^{nd}$ such that $F_k(\tilde{x})=\nu_k$ for all $1\leq k\leq d$, and $v_p(J_0)=\tau<\infty$ for some $d\times d$ minor $J_0$ of the Jacobian $J_F(\tilde{x})$. Then, we have 
    \[
    \mathfrak{S}_p(\nu)\geq p^{-d(2\tau+1)(n-1)}.
    \]
\end{lemma}

\begin{proof}
    Without loss of generality, assume $J_0=\det B$, where $B=(\partial F_k/\partial x_\ell (\tilde{x}))_{1\leq k,\ell\leq d}$. Let us write the variable as $x=(u,w)$ with $u\in\Z_p^d$ and $w\in \Z_p^{nd-d}$, and write $\tilde{x}=(\tilde{u},\tilde{w})$. Let $w\in \Z_p^{nd-d}$ with $w\equiv \tilde{w}\pmod {p^{2\tau+1}}$, and define 
    \[
    f_w(u)=F(u,w)-\nu=(F_1(u,w)-\nu_1,\dots,F_d(u,w)-\nu_d)\in \Z_p[u]^d.
    \]
    Since $F_k(\tilde{x})=F_k(\tilde{u},\tilde{w})=\nu_k$, we have 
    \[
    f_w(\tilde{u})=F(\tilde{u},w)-\nu=(F_1(\tilde{u},w)-F_1(\tilde{u},\tilde{w}), \dots, F_d(\tilde{u},w)-F_d(\tilde{u},\tilde{w})).
    \]

    Since $w-\tilde{w}\in p^{2\tau+1}\Z_p^{nd-d}$, it follows that 
    \[
    f_w(\tilde{u})\equiv 0\pmod {p^{2\tau+1}}, \quad J_{f_w}(\tilde{u})\equiv B \pmod {p^{2\tau+1}},
    \]
    where $J_{f_w}$ is the Jacobian of the variety given by $f_w=0$. Since $J_{f_w}(\tilde{u})\equiv B \pmod {p^{2\tau+1}}$, their determinant must be congruent mod $p^{2\tau+1}$, i.e., $\det J_{f_w}(\tilde{u})\equiv J_0\pmod {p^{2\tau+1}}$. By our assumption $v_p(J_0)=\tau<2\tau+1$, we get that $v_p(\det J_{f_w}(\tilde{u}))=\tau$. By generalized Hensel's lemma (see Lemma \ref{genhensel}), there exists a unique $u(w)\in\Z_p^d$ such that $f_w(u(w))=0$ and $u(w)\equiv \tilde{u}\pmod {p^{\tau+1}}$, i.e., $F_k(u(w),w)=\nu_k$ and $u(w)\equiv \tilde{u}\pmod {p^{\tau+1}}$ for all $k$.

    Now, fix an integer $N\geq 2\tau+1$, and define the set 
    \[
    \mathcal{W}=\{w\in (\Z/p^N\Z)^{nd-d}\mid w\equiv \tilde{w} \pmod {p^{2\tau+1}}\}.
    \]
    Observe that $|\mathcal{W}|=p^{(N-(2\tau+1))(nd-d)}$; indeed, each $nd-d$ coordinate of $w$ can be chosen freely in $\Z/p^N\Z$ subject to mod $p^{2\tau+1}$, giving $p^{N-(2\tau+1)}$ possibilities per coordinate. 

    Each $w\in \mathcal{W}$ can be lifted to an element in $\Z_p^{nd-d}$, still denoted by $w$, satisfying the same congruence. Then, by the above argument, we get a $\Z_p$ point $x(w)=(u(w),w)$ such that $F_k(u(w),w)=\nu_k$ for all $k$. Reducing this modulo $p^N$, we have $F_k(x)\equiv \nu_k\pmod {p^N}$ for all $k$. Note that two distinct $w,w'\in \mathcal{W}$ give two distinct solutions, as their $w$-coordinates differ. Thus, we have 
    \[
     \mathcal{A}_p(N;\nu)=\#\{x\in (\Z/p^N\Z)^{nd}\mid F_k(x)\equiv \nu_k\pmod {p^N} \, \forall 1\leq k\leq d\}\geq p^{(N-(2\tau+1))(nd-d)}.
    \]
    By Theorem \ref{birchthm}\,(3), 
    \[
    \mathfrak{S}_p(\nu)=\lim_{N\rightarrow\infty}p^{-Nd(n-1)} \mathcal{A}_p(N;\nu)\geq  p^{-d(2\tau+1)(n-1)},
    \]
    as we wanted.
\end{proof}

Next, using the previous two lemmas, we establish a connection between local representation by a diagonal $m$-ic form over $K$ and minors with bounded $p$-adic valuations.

\begin{proposition}\label{pro:6.3}
    Let $p$ be a rational prime, $m\geq 2$ be an even integer, $K$ a totally real number field of degree $d$, $\alpha\in \co_K^+$, and $Q(x)=\sum_{i=1}^n a_ix_i^m$ be a totally positive definite diagonal $m$-ic form over $K$. Let $V_F(\alpha)$ be the variety defined in \eqref{eq:VF} corresponding to $\alpha$.

    Assume that for every prime ideal $\mathfrak{p}|p$, there exists $x^\mathfrak{p}=(x_1^\mathfrak{p},\dots,x_n^\mathfrak{p})\in \co_\mathfrak{p}^n$ such that $Q(x^\mathfrak{p})=\alpha$, and put 
    \[
    s_\mathfrak{p}=\min_{1\leq i\leq n}v_p(m a_i (x_i^\mathfrak{p})^{m-1}).
    \]
    Then, the variety $V_F(\alpha)$ possesses a point $\tilde{x}\in\Z_p^{nd}$ such that the ideal generated by all $d\times d$ minors of $J_F(\tilde{x})$ is $p^T\Z_p$, where 
    \[
    T=\sum_{\mathfrak{p}|p}f(\mathfrak{p}|p)s_\mathfrak{p},
    \]
    and $f(\mathfrak{p}|p)$ denotes the inertia degree. In particular, there exists a $d\times d$ minor of $J_F(\tilde{x})$ with $p$-adic valuation exactly $T$.
\end{proposition}

\begin{proof}
    Let us first establish that $V_F(\alpha)$ has a $\Z_p$ point. Recall the natural isomorphism of $\Z_p$-algebra 
    \[
    \phi: \co_K \otimes_{\Z} \Z_p \rightarrow \prod_{\mathfrak{p}|p} \co_\mathfrak{p}: y\otimes 1\mapsto (\iota_\mathfrak{p}(y))_{\mathfrak{p}|p},
    \]
    where $\iota_\mathfrak{p}: \co_K\hookrightarrow \co_\mathfrak{p}$ is the canonical embedding.
    By our assumption, for each $\mathfrak{p}|p$, there exists a $x^\mathfrak{p}=(x_1^\mathfrak{p}, \dots, x_n^\mathfrak{p})\in \co_\mathfrak{p}^n$ such that $Q(x^{\mathfrak{p}})=\alpha$ in $\co_\mathfrak{p}$. For each $1\leq i\leq n$, the tuple $(x_i^\mathfrak{p})_{\mathfrak{p}|p}$ is an element of $\prod_{\mathfrak{p}|p}\co_\mathfrak{p}$. Since $\phi$ is an isomorphism, there exists a unique $\tilde{x_i}\in  \co_K \otimes_{\Z} \Z_p$ such that 
    \[
    \phi(\tilde{x_i})=(x_i^\mathfrak{p})_{\mathfrak{p}|p}.
    \]
    Define $\tilde{x}=(\tilde{x_1},\dots,\tilde{x_n})\in (\co_K \otimes_{\Z} \Z_p)^n$. We extend $Q$ to $\co_K \otimes_{\Z} \Z_p$ by 
    \[
    Q(\tilde{x})=\sum_{i=1}^n (a_i\otimes 1)\tilde{x_i}^m.
    \]
    Applying $\phi$, we get 
    \[
    \phi(Q(\tilde{x}))=\sum_{i=1}^n \phi(a_i\otimes 1) \phi(\tilde{x_i})^m=\sum_{i=1}^n(\iota_\mathfrak{p}(a_i))_{\mathfrak{p}|p}(x_i^\mathfrak{p})_{\mathfrak{p}|p}^m.
    \]
    The multiplication in the product ring is componentwise, so
    \[
    \phi(Q(\tilde{x}))=\left(\sum_{i=1}^n \iota_\mathfrak{p}(a_i)(x_i^\mathfrak{p})^m\right)_{\mathfrak{p}|p}=(\iota_\mathfrak{p}(Q(x^\mathfrak{p})))_{\mathfrak{p}|p}.
    \]
    As $Q(x^{\mathfrak{p}})=\alpha$ in $\co_\mathfrak{p}$, we have 
    \[
     \phi(Q(\tilde{x}))=(\iota_\mathfrak{p}(Q(x^\mathfrak{p})))_{\mathfrak{p}|p}=(\iota_\mathfrak{p}(\alpha))_{\mathfrak{p}|p}=\phi(\alpha\otimes 1).
    \]
    Since $\phi$ is injective, $Q(\tilde{x})=\alpha\otimes 1$ holds in $\co_K\otimes_{\Z} \Z_p$. The $\Z_p$-module  $\co_K\otimes_{\Z} \Z_p$ is free of rank $d$ with a basis 
    \[
    u_j:=\omega_j\otimes 1\quad \forall\, 1\leq j\leq d,
    \]
    where $\omega_1,\dots,\omega_d$ is an integral basis from Proposition \ref{unitadjust}.
    We write each $\tilde{x_i}$ in this basis as
    \[
    \tilde{x_i}=\sum_{j=1}^d \tilde{x}_{i,j} u_j, \quad \tilde{x}_{i,j}\in\Z_p.
    \]
    Now, compute $Q(\tilde{x})$ by expanding everything in the basis $u_1,\dots,u_d$, and using the properties of the tensor product, we obtain 
    \[
    \sum_{i=1}^n (a_i\otimes 1)\left(\sum_{j=1}^d \tilde{x}_{i,j} u_j\right)^m=\sum_{k=1}^d F_k(\tilde{x})u_k,
    \]
    where $F_k$ are the same polynomials as in Lemma \ref{lem:5.1}, now evaluated at $\tilde{x}_{i,j}\in\Z_p$. On the other hand, the image of $\alpha$ in $\co_K\otimes_{\Z}\Z_p$ is
    \[
    \alpha\otimes 1=\left(\sum_{k=1}^d \alpha_k \omega_k\right)\otimes 1=\sum_{k=1}^d \alpha_k (\omega_k\otimes 1)=\sum_{k=1}^d \alpha_k u_k.
    \]
    Since $Q(\tilde{x})=\alpha$ and $u_k$ are a $\Z_p$-basis, we get the equality
    \[
    F_k(\tilde{x})=\alpha_k \quad \forall \, 1\leq k\leq d.
    \]
    Hence, we obtained a $\tilde{x}\in \Z_p^{nd}$ on $V_F(\alpha)$. 

    Consider the polynomial identity over $\co_K\otimes_{\Z}\Z_p$
    \begin{equation}\label{e:1.1modp}
    \sum_{i=1}^n (a_i\otimes 1)\left(\sum_{j=1}^d x_{i,j}u_j \right)^m
    =\sum_{k=1}^d F_k(x)u_k.
    \end{equation}
    This identity holds because both sides define the same polynomial when evaluated on
    elements of $\Z_p$.
    Differentiating \eqref{e:1.1modp} with respect to $x_{i,j}$ and evaluating at
    $x=\tilde{x}$, we obtain
    \begin{equation}\label{e:1.2modp}
    \begin{aligned}
    \frac{\partial}{\partial x_{i,j}}\left((a_i\otimes 1)\Bigl(\sum_{r=1}^dx_{i,r}u_r\Bigr)^m\right)
    &= m\,(a_i\otimes 1)\Bigl(\sum_{r=1}^d\tilde{x}_{i,r}u_r\Bigr)^{m-1} u_j \\
    &= \sum_{k=1}^d \frac{\partial F_k}{\partial x_{i,j}}(\tilde{x}) u_k.
    \end{aligned}
    \end{equation}

    Since  $\tilde{x}_i=\sum_{r=1}^d \tilde{x}_{i,r}u_r$,
     \eqref{e:1.2modp} becomes
    \begin{equation}\label{e:1.3modp}
    m\,(a_i\otimes 1)\,\tilde{x}_i^{\,m-1}u_j
       =\sum_{k=1}^d \frac{\partial F_k}{\partial x_{i,j}}(\tilde{x})\,u_k .
    \end{equation}
    Define
    \[
    c_i := m\,(a_i\otimes 1)\,\tilde{x}_i^{\,m-1}\;\in\;\co_K\otimes_{\Z}\Z_p .
    \]
    Then, the left‑hand side of \eqref{e:1.3modp} is $c_i u_j$.
    Now expand $c_i u_j$ in the basis $u_1,\dots,u_d$, i.e., 
    \[
    c_i u_j = \sum_{k=1}^d C_{i,j,k}\,u_k ,\qquad C_{i,j,k}\in\Z_p .
    \]
    Substituting this into \eqref{e:1.3modp} and comparing coefficients gives
    \[
    \frac{\partial F_k}{\partial x_{i,j}}(\tilde{x})=C_{i,j,k}
    \qquad \forall\, 1\le i\le n,\;\forall\, 1\le j,k\le d.
    \]

    For a fixed $1\le i\le n$, the partial derivatives with respect to the variables
    $x_{i,1},\dots,x_{i,d}$ form a $d\times d$ matrix $J_i$, whose $(k,j)$-entry is
    $C_{i,j,k}$.  By construction, $J_i$ is exactly the matrix of multiplication by $c_i$
    on the $\Z_p$-algebra $\co_K\otimes_{\Z}\Z_p$ with respect to the basis $u_1,\dots, u_d$.
    The full Jacobian is the concatenation
    \[
    J_F(\tilde{x})=\bigl(J_1\mid J_2\mid\cdots\mid J_n\bigr).
    \]

    Identifying $\co_K\otimes_{\Z}\Z_p\cong\Z_p^{\,d}$, the Jacobian $J_F(\tilde{x})$
    defines a $\Z_p$-linear map
    \[
    f: (\co_K\otimes_{\Z}\Z_p)^n \rightarrow \co_K\otimes_{\Z}\Z_p:
    (y_1,\dots,y_n)\longmapsto \sum_{i=1}^n c_i y_i.
    \]
    The image of $f$ is precisely the ideal generated by $c_i$
    \[
    I = \langle c_1,c_2,\dots,c_n\rangle \subseteq \co_K\otimes_{\Z}\Z_p .
    \]
    Then, we have 
    \[
    \phi(I)=\sum_{i=1}^n\phi(c_i)\cdot \prod_{\mathfrak{p}|p}\co_\mathfrak{p}=\prod_{\mathfrak{p}|p}I_\mathfrak{p},
    \]
    where $I_\mathfrak{p}$ is the ideal generated by the image of $c_i$ under the projection to $\co_\mathfrak{p}$. More precisely, $\phi(c_i)=(m a_i (x_i^\mathfrak{p})^{m-1})_{\mathfrak{p}|p}$. Therefore, for a fixed $\mathfrak{p}|p$, 
    \[
    I_\mathfrak{p}=\langle m a_i (x_i^\mathfrak{p})^{m-1} \mid 1\leq i\leq n\rangle
    \]
    in $\co_\mathfrak{p}$. Since $\co_\mathfrak{p}$ is a discrete valuation ring, $I_\mathfrak{p}=\mathfrak{p}^{s_\mathfrak{p}}\co_\mathfrak{p}$. We now compute the index of $I$ in $\co_K\otimes_\Z\Z_p$. Using the isomorphism $\phi$, we get 
    \[
    [\co_K\otimes_\Z\Z_p:I]= \left[\prod_{\mathfrak{p}|p} \co_\mathfrak{p}:  \prod_{\mathfrak{p}|p}I_\mathfrak{p}\right]=\prod_{\mathfrak{p}|p} [\co_\mathfrak{p}: \mathfrak{p}^{s_\mathfrak{p}}\co_\mathfrak{p}]=\prod_{\mathfrak{p}|p}p^{f(\mathfrak{p}|p)s_\mathfrak{p}}=p^{\sum_{\mathfrak{p}|p}f(\mathfrak{p}|p)s_\mathfrak{p}}=p^T.
    \]
    This shows that $J_F(\tilde{x})\Z_p^{nd}=I$ has index $p^T$ in $\Z_p^d$. By Lemma \ref{lem:6.1}, the proposition follows.
\end{proof}

\section{Asymptotic local--global principle}

In this section, we prove the asymptotic local--global principle (Corollary \ref{cor:main}) using Theorem \ref{birchthm}. To this end, we utilize the notions and results developed in the previous three sections. Note that we have not yet proven that the singular integral is uniformly bounded away from zero. We establish this while proving the asymptotic local--global principle. 

\begin{theorem}\label{asymlocalglobal}
    Let $m\geq 2$ be an even integer, $K$ be a totally real number field of degree $d$, and $Q=\sum_{i=1}^na_ix_i^m$ be a totally positive definite diagonal $m$-ic form in $n$ variables over $K$ with $n>d(d+1)(m-1)2^{m-1}$.
    
    Assume that there exists an integer $t\geq 0$ such that for every $\alpha\in\co_K^+$ and for every rational prime $p$ there is $\tilde{x}\in\Z_p^{nd}$ on $V_F(\alpha)$ such that some $d\times d$ minor of $J_F(\tilde{x})$ has $p$-adic valuation at most $t$, where $V_F(\alpha)$ is the variety defined in \eqref{eq:VF} and $J_F(\tilde{x})$ is its Jacobian evaluated at $\tilde{x}$.

    Then, there exists a constant $C>0$ depending only on $K,m,Q$, and $t$ such that every $\alpha\in\co_K^+$ with $\N(\alpha)>C$ is represented by $Q$.
\end{theorem}

\begin{proof}
 Throughout, we work with an integral basis $\{\omega_1, \omega_2,\dots, \omega_d\}$ from Proposition \ref{unitadjust}. By Lemma \ref{lem:equivalentsystem}, $Q$ representing $\alpha\in\co_K^+$ is equivalent to the representation of coordinates of $\alpha$ (with respect to this integral basis) by corresponding integral $m$-ic forms $F_k$ defined in Lemma \ref{lem:5.1}. To this end, we apply Theorem \ref{birchthm}. Let us first verify the hypothesis of Theorem \ref{birchthm}. 
 By Proposition \ref{singdim}, $\dim V^\ast=nd-n$, so $K_0=n2^{-(m-1)}>d(d+1)(m-1)$ by our assumption.

 Let $\alpha\in\co_K^+$. By Lemma \ref{lem:4.1}, there exists $\varepsilon\in\co_K^\times$ such that $\beta:=\alpha\varepsilon^m$ satisfies 
 \begin{equation}\label{e:7.1}
     c_1'\N(\alpha)^{1/d}\leq \sigma_i(\beta)\leq c_2'\N(\alpha)^{1/d},
 \end{equation}
 for all $1\leq i\leq d$, where $\sigma_1,\dots,\sigma_d$ are real embeddings of $K$, and $c_1',c_2'>0$ are constants from Lemma \ref{lem:4.1}. Writing $\beta=\sum_{k=1}^d\beta_k\omega_k$ with $\beta_k\in\Z$; from Proposition \ref{unitadjust}, it follows that 

 \begin{equation}\label{e:7.2}
     c_1\N(\alpha)^{1/d}\leq \beta_k\leq c_2\N(\alpha)^{1/d}
 \end{equation}
for all $1\leq k\leq d$, where $c_1,c_2>0$ are constants from Proposition \ref{unitadjust}. Note that $\beta\in\co_K^+$ and $\N(\beta)=\N(\alpha)$. Let us write $\bar{\beta}=(\beta_1,\dots,\beta_d)\in\Z^d$.

Set 
\[
\beta_{max}=\max_{1\leq k\leq d}\beta_k.
\]
Choose a constant $Y\geq 1$ (to be determined later) depending only on $K,m,Q$, and define 
\[
P=\lceil (Y\beta_{max})^{1/m}\rceil.
\]
Since $Y\beta_{max}\geq c_1\N(\alpha)^{1/d}$, we have 
\begin{equation}\label{e:7.3}
    (Y\beta_{max})^{1/m}\leq P \leq 2 (Y\beta_{max})^{1/m}.
\end{equation}

Now, set $\mu=\bar{\beta}/P^m\in\R^d$. For each embedding $\sigma_r$, we let 
\[
\lambda_r(\mu)=\sum_{k=1}^d\mu_k\sigma_r(\omega_k)=\frac{\sigma_r(\beta)}{P^m}.
\]
From \eqref{e:7.1}--\eqref{e:7.3} we obtain, for each $k,r$,
\begin{equation}\label{e:7.4}
    \frac{c_1}{2^mc_2Y}\leq \mu_k\leq \frac{1}{Y}, \quad \frac{c_1'}{2^mc_2Y}\leq \lambda_r(\mu)\leq \frac{c_2'}{c_1Y}
\end{equation}
Thus, $\mu$ lies in the compact set 
\[
\mathcal{K}_Y=\left\{\mu\in \R^d \Bigg|  \frac{c_1}{2^mc_2Y}\leq \mu_k\leq \frac{1}{Y}, \quad \frac{c_1'}{2^mc_2Y}\leq \lambda_r(\mu)\leq \frac{c_2'}{c_1Y}, \, \forall\, 1\leq k,r\leq d\right\}.
\]
Note that $\mathcal{K}_Y$ is independent of $\alpha$. Let $\Delta=(\sigma_i(\omega_j))_{1\leq i,j\leq d}$; this is an invertible matrix. Observe that the system $F_k(x)=\mu_k$ is equivalent, via the invertible change of variables 
\[
y_{i,r}=\sum_{j=1}^d\sigma_r(\omega_j)x_{i,j}
\]
to the system 
\begin{equation}\label{e:7.5}
    \sum_{i=1}^n\sigma_r(a_i)y_{i,r}^m=\lambda_r(\mu),
\end{equation}
for all $1\leq r\leq d$. For $\mu\in\mathcal{K}_Y$, we have $\lambda_r(\mu)>0$; thus, we can define 
\[
y_{1,r}(\mu)=\left(\frac{\lambda_r(\mu)}{\sigma_r(a_1)} \right)^{1/m}>0, \quad y_{i,r}(\mu)=0, \, \forall \,2\leq i\leq n.
\]

We use $y_{i,r}(\mu)$ to emphasize the dependency on $\mu$. Let $x^\mu\in\R^{nd}$ be the point corresponding to these $y$'s under the inverse change of variables. Then, we have $F_k(x^\mu)=\mu_k$ for all $1\leq k\leq d$. Set 
\[
a =\min_{1\leq r\leq d}\sigma_r(a_1)>0.
\]
Writing $\Delta^{-1}=(d_{i,j})_{1\leq i,j\leq d}$, let 
\[
\|\Delta^{-1}\|_\infty:=\max_{1\leq i\leq d}\sum_{j=1}^d |d_{i,j}|.
\]
From \eqref{e:7.4}, we obtain 
\[
y_{1,r}(\mu)\leq \left(\frac{c_2'}{c_1aY} \right)^{1/m}.
\]
We now define the constant
\[
Y=\max\left\{1, \frac{c_2'(4\|\Delta^{-1}\|_\infty)^m}{c_1a}\right\}.
\]
Then, for every $\mu\in\mathcal{K}_Y$, we have 
\begin{equation}\label{e:7.6}
    \|x^\mu\|_\infty\leq \|\Delta^{-1}\|_\infty \max_{1\leq r\leq d} y_{1,r}(\mu)\leq \frac{1}{4}.
\end{equation}

Take the box $\mathcal{B}_0=[-1/2, 1/2]^{nd}\subset[-1,1]^{nd}$. By \eqref{e:7.6}, all points $x^\mu$ lie in $\operatorname{int} \mathcal{B}_0$. Moreover, since $y_{1,r}(\mu)>0$ for all $r$, the point $x^\mu\notin L_r$ for all $r$, where $L_r$ is as defined in the proof of Proposition \ref{singdim}. Since $V^\ast=\cup_{r=1}^dL_r$, we have 
\[
\mathcal{C}=\{x^\mu\mid \mu\in \mathcal{K}_Y\}\subset \operatorname{int}\mathcal{B}_0\setminus V^\ast.
\]
Since the map $\mu\mapsto x^\mu$ is continuous on the compact set $\mathcal{K}_Y$, the set $\mathcal{C}$ is compact. From Theorem \ref{birchthm} (1), there exists $c_3>0$ such that 
\[
\Psi(\mu)\geq c_3,
\]
for all $\mu\in\mathcal{K}_Y$. This gives a uniform bound on the singular integral. 

Now, for singular series, consider $\bar{\beta}\in\Z^d$. By our assumption, for every $p$ there exists $\tilde{x}\in\Z_p^{nd}$ on $V_F(\bar{\beta})$ at which some $d\times d$ minor of $J_F(\tilde{x})$ has $p$-adic valuation at most $t$. Lemma \ref{lem:6.2} then gives 
\[
\mathfrak{S}_p(\bar{\beta})\geq p^{-d(2t+1)(n-1)},
\]
for all $p$. By Theorem \ref{birchthm} (2), it follows that 
\[
\mathfrak{S}(\bar{\beta})=\prod_{p}\mathfrak{S}_p(\bar{\beta})\geq \frac{1}{2}\prod_{p\leq p_0}p^{-d(2t+1)(n-1)}=:c_4>0.
\]

Note that $\bar{\beta}=P^m\mu$ with $\mu\in\mathcal{K}_Y$. Applying Theorem \ref{birchthm} to forms $F_k$, the box $\mathcal{B}_0$, and $\nu=\bar{\beta}$ yields, for all sufficiently large $P$,
\[
M(P;\bar{\beta})=\mathfrak{S}(\bar{\beta})\Psi(\mu)P^{d(n-m)}+O(P^{d(n-m)-\delta}),
\]
where $\delta>0$. Using bounds on the singular integral and the singular series obtained above, we have 
\[
M(P;\bar{\beta})\geq c_3c_4P^{d(n-m)}-c_0P^{d(n-m)-\delta}.
\]
From \eqref{e:7.2} and \eqref{e:7.3}, we have $P^m\geq Y\beta_{max}\geq Yc_1\N(\alpha)^{1/d}$. Thus, $P\rightarrow \infty$ as $\N(\alpha)\rightarrow\infty$. Choose $C>0$ such that $\N(\alpha)>C$ forces 
\[
c_3c_4P^{d(n-m)}-c_0P^{d(n-m)-\delta}>0;
\]
then $M(P;\bar{\beta})>0$. Consequently, there exists $x=(x_{1,1},\dots,x_{n,d})\in\Z^{nd}$ such that $F_k(x)=\beta_k$ for all $1\leq k\leq d$. 

For $1\leq i\leq n$, define $x_i=\sum_{j=1}^dx_{i,j}\omega_j\in\co_K$. By Lemma \ref{lem:equivalentsystem}, $Q(x_1,\dots,x_n)=\beta=\alpha\varepsilon^m$. Finally, set $y_i=\varepsilon^{-1}x_i$, then $Q(y_1,\dots,y_n)=\varepsilon^{-m}\beta=\alpha$. Thus, $\alpha\in\co_K^+$ with $\N(\alpha)>C$ is represented by $Q$, completing the proof.
\end{proof}

It is straightforward to see that Theorem \ref{asymlocalglobal} remains valid under the assumption of Proposition \ref{pro:6.3}. We record this as a corollary for later use.

\begin{corollary}\label{cor:alg}
     Let $m\geq 2$ be an even integer, $K$ be a totally real number field of degree $d$, and $Q=\sum_{i=1}^na_ix_i^m$ be a totally positive definite diagonal $m$-ic form in $n$ variables over $K$ with $n>d(d+1)(m-1)2^{m-1}$.
    
    Assume that there exists an integer $t\geq 0$ such that for every $\alpha\in\co_K^+$ there exists $x\in\co_\mathfrak{p}^n$ with $Q(x)=\alpha$ and 
    \[
    \min_{1\leq i\leq n}v_\mathfrak{p}(ma_ix_i^{m-1})\leq t,
    \]
    for every prime ideal $\mathfrak{p}\subset\co_K$.
    
     Then, there exists a constant $C>0$ such that every $\alpha\in \co_K^+$ with $\N(\alpha)>C$ is represented by $Q$.
\end{corollary}

\begin{proof}
    Follows immediately from Proposition \ref{pro:6.3} and Theorem \ref{asymlocalglobal}. 
\end{proof}

\begin{rem}
 Note that in the above proof, after applying Proposition \ref{pro:6.3}, the constant $t$, and hence the constant $C$, may differ from those in Theorem \ref{asymlocalglobal}. However, for notational simplicity, we use the same notation. 
\end{rem}

With this, we have established the Corollary \ref{cor:main} stated in the introduction. We now proceed to establish the main theorem.

\section{Local universality}\label{S8}

In this section, we will show that the local assumption of Corollary \ref{cor:alg} is satisfied by forms produced by Algorithm \ref{alg:escalation}. In this context, we will first show that a totally positive definite diagonal $m$-ic form with sufficiently many variables (in fact, three variables suffice) represents all of $\co_\mathfrak{p}$ for all but finitely many $\mathfrak{p}$. This extends \cite[92:1b]{om} to diagonal forms of higher degree. To handle the remaining finitely many $\mathfrak{p}$'s, we need to exploit local arithmetic to obtain a uniform bound on the valuation stated in Corollary \ref{cor:alg}.

To establish local representability outside a finite set of primes, we require the following result, which ultimately follows from the work of Weil \cite{weil}. Although Weil's result applies in considerably greater generality than is needed here; specializing his argument to diagonal forms yields precisely the statement below. Since this requires only a routine adaptation of his proof, we omit the details to avoid unnecessary technicalities.

\begin{theorem}[{\cite{weil}}]\label{thm:weil}
    Let $p$ be a rational prime, $\mathbb F_q$ be the finite field of characteristic $p$. Fix an integer $m\geq 2$ and non-zero elements $a_1,a_2,\dots,a_k\in\mathbb F_q^\times$. Then, there exists a constant $T$ depending only on $m,\,k$, such that if $q>T$ and $p\nmid m$, then for every $b\in \mathbb F_q$ the equation 
    \[
    \sum_{i=1}^k a_ix_i^m=b
    \]
    has a non-zero solution $(x_1,\dots,x_k)\in\mathbb F_q^k$. Moreover, when $k \geq 3$, the constant $T$ can be taken to be 
    \[
    T=\left(2(m-1)^3+2(m-1)^2+1 \right)^2.
    \]
\end{theorem}

Using this, we can now establish local universality outside a finite set of primes. 

\begin{proposition}\label{prop:8.2}
    Let $m\geq 2$ be an even integer, $K$ be a totally real number field, and $Q=\sum_{i=1}^n a_ix_i^m$ be a totally positive definite diagonal $m$-ic form over $K$. Assume $n\geq 3$. Then, for every $\alpha\in\co_K^+$ the following hold: 
    \begin{enumerate}
        \item There exists a finite set of prime ideals $S$ depending only on $m,a_1,a_2,a_3$, such that the equation $Q(x)=\alpha$ has a solution in $\co_{\mathfrak p}^n$ for every prime ideal $\mathfrak p\notin S$.

        \item The solution in (1) can be chosen such that 
        \[
        \min_{1\leq i\leq n} v_\mathfrak{p}(m a_ix_i^{m-1})=0,
        \]
        for all $\mathfrak{p}\notin S$.
    \end{enumerate}
\end{proposition}

\begin{proof}
    It suffices to prove the result for $n=3$, since any additional variables can simply be set equal to $0$.
    
    Define $S$ to be the set of prime ideals $\mathfrak{p}\subset \co_K$ satisfying at least one of the following:
    \begin{itemize}
        \item $\mathfrak{p}\mid m$,
        \item $\mathfrak{p}\mid a_1a_2a_3$,
        \item the norm $q=\N(\mathfrak{p})=|\co_K/\mathfrak{p}|\leq T=\left(2(m-1)^3+2(m-1)^2+1 \right)^2$.
    \end{itemize}
    It is clear that $S$ is a finite set. It depends on $m,a_1,a_2,a_3$, but not on later coefficients of $Q$. Take $\mathfrak{p}\notin S$ and set $\mathbb F_q=\co_K/\mathfrak{p}$. Then, it follows that 
    \begin{itemize}
        \item $q=\N(\mathfrak{p})>T$,
        \item $p=\operatorname{char}\mathbb F_q$ does not divide $m$,
        \item the residues of $a_
        1,a_2,a_3$ in $\mathbb F_q$ are non-zero.
    \end{itemize}

    Let $\alpha\in\co_K^+$ be arbitrary and $b$ be its reduction modulo $\mathfrak{p}$. By Theorem \ref{thm:weil}, applied with $k=3$, the congruence 
    \[
    a_1x_1^m+a_2x_2^m+a_3x_3^m\equiv b\pmod {\mathfrak{p}}
    \]
    possesses a non-zero solution $(\bar{x_1},\bar{x_2},\bar{x_3})\in\mathbb{F}_q^3$. Choose an index $1\leq j\leq 3$ with $\bar{x_j}\neq 0$. Since $\mathfrak{p}\nmid m$ and $\mathfrak{p}\nmid a_j$,
    \[
    \frac{\partial}{\partial x_j}\left(\sum_{i=1}^3a_ix_i^m \right)=ma_jx_j^{m-1}
    \]
    evaluated at the solution is non-zero modulo $\mathfrak{p}$; hence it is a unit in $\mathbb F_q$. Now, consider the polynomial 
    \[
    F(x_1,x_2,x_3)=\sum_{i=1}^3a_ix_i^m-\alpha\in\co_K[x_1,x_2,x_3].
    \]
    We have already found a solution modulo $\mathfrak{p}$ that is non-singular for the variable $x_j$, i.e., 
    \[
    F(\bar{x}_1,\bar{x}_2,\bar{x}_3)\equiv0\pmod{\mathfrak{p}}, \quad \frac{\partial  F}{\partial x_j}(\bar{x}_1,\bar{x}_2,\bar{x}_3)\not\equiv 0\pmod{\mathfrak{p}}.
    \]
    By Hensel's lemma, the non-singular solution lifts uniquely to a solution $x=(x_1,x_2,x_3)\in\co_\mathfrak{p}^3$ of $F(x)=0$ establishing (1).

    Since lifting preserves the residue classes, so 
    \[
    \frac{\partial F}{\partial x_j}(x_1,x_2,x_3)\equiv ma_jx_j^{m-1}\pmod {\mathfrak{p}\co_{\mathfrak{p}}},
    \]
    which is a unit in $\co_\mathfrak{p}$. Consequently, the minimum 
    \[
    \min_{1\leq i\leq n} v_\mathfrak{p}(m a_ix_i^{m-1})=0,
    \] giving (2).
\end{proof}

\begin{lemma}\label{lem:8.3}
    Let $m\geq 2$ be an even integer, $K$ a totally real number field, and $\mathfrak{p}\subset\co
    _K$ be a prime ideal. Set 
    \[
    e_\mathfrak{p}=v_\mathfrak{p}(m), \quad k_\mathfrak{p}=2e_\mathfrak{p}+1.
    \]
    Then, the following hold: 
    \begin{enumerate}
        \item We have, $1+\mathfrak{p}^{k_\mathfrak{p}}\co_\mathfrak{p}\subseteq \co_\mathfrak{p}^{\times m}$.
        \item The group $K_\mathfrak{p}^\times/K_\mathfrak{p}^{\times m}$ is finite; moreover,
        \[
        |K_\mathfrak{p}^\times/K_\mathfrak{p}^{\times m}|\leq m |(\co_\mathfrak{p}/\mathfrak{p}^{k_\mathfrak{p}})^\times|.
        \]
        \item The valuation $v_\mathfrak{p}$ induces a well-defined map $\bar{v}_\mathfrak{p}:K_\mathfrak{p}^\times/K_\mathfrak{p}^{\times m}\rightarrow \Z/m\Z$. Let 
        \[
        B_\mathfrak{p}=\{[g]\in K_\mathfrak{p}^\times/K_\mathfrak{p}^{\times m}\mid [g]\cap (\co_\mathfrak{p}\setminus\{0\})\neq \emptyset\}.
        \]
        For every $[g]\in B_\mathfrak{p}$, the minimum
        \[
        r_g=\min\{v_\mathfrak{p}(\alpha)\mid \alpha\in [g]\cap \co_\mathfrak{p}\}
        \]
        exists, is attained, and satisfies $0\leq r_g\leq m-1$.
        
    \end{enumerate}
\end{lemma}

\begin{rem}
    We use the notation $[g]$ to denote the class in the field, distinguishing it from the class in the ring.
\end{rem}

\begin{proof}
    Let $a\in 1+\mathfrak{p}^{k_\mathfrak{p}}\co_\mathfrak{p}$, and consider the polynomial 
    \[
    f(x)=x^m-a\in\co_\mathfrak{p}[x].
    \]
    Then, $f(1)=1-a\in \mathfrak{p}^{k_\mathfrak{p}}\co_\mathfrak{p}$, so $v_\mathfrak{p}(f(1))\geq k_\mathfrak{p}$. The derivative $f'(x)=mx^{m-1}$ satisfies $f'(1)=m$, so $v_\mathfrak{p}(f'(1))=v_\mathfrak{p}(m)=e_\mathfrak{p}$. Thus, we have 
    \[
    v_\mathfrak{p}(f(1))\geq 2e_\mathfrak{p}+1>2e_\mathfrak{p}=2v_\mathfrak{p}(f'(1)).
    \]
    Therefore, by Hensel's lemma, there exists $b\in\co_\mathfrak{p}$ such that $f(b)=0$ and 
    \[
    v_\mathfrak{p}(b-1)\geq v_\mathfrak{p}(f(1))-v_\mathfrak{p}(f'(1))\geq (2e_\mathfrak{p}+1)-e_\mathfrak{p}=e_\mathfrak{p}+1.
    \]
    In particular, $b\equiv 1\pmod{\mathfrak{p}^{e_\mathfrak{p}+1}}$, so $b\in \co_\mathfrak{p}^\times$. Therefore, $a=b^m\in\co_\mathfrak{p}^{\times m}$ and hence $1+\mathfrak{p}^{k_\mathfrak{p}}\co_\mathfrak{p}\subseteq \co_\mathfrak{p}^{\times m}$, establishing (1). 

    \medskip

    Consider the endomorphism
    \[
    \co_\mathfrak{p}^\times\rightarrow \co_\mathfrak{p}^\times: x\mapsto x^m.
    \]
    By (1), the image of this automorphism contains the subgroup $1+\mathfrak{p}^{k_\mathfrak{p}}\co_\mathfrak{p}$. Thus, we have a surjection 
    \[
    \co_\mathfrak{p}^\times/1+\mathfrak{p}^{k_\mathfrak{p}}\co_\mathfrak{p} \rightarrow \co_\mathfrak{p}^\times/\co_\mathfrak{p}^{\times m}.
    \]
    Note that $\co_\mathfrak{p}/\mathfrak{p}^{k_\mathfrak{p}}$ is finite, therefore the group $(\co_\mathfrak{p}/\mathfrak{p}^{k_\mathfrak{p}})^\times$ is also finite, and the canonical map $\co_\mathfrak{p}^\times\rightarrow (\co_\mathfrak{p}/\mathfrak{p}^{k_\mathfrak{p}})^\times$ is surjective with kernel $1+\mathfrak{p}^{k_\mathfrak{p}}\co_\mathfrak{p}$. Then, by the first isomorphism theorem, it follows that $\co_\mathfrak{p}^\times/1+\mathfrak{p}^{k_\mathfrak{p}}\co_\mathfrak{p}$ is finite, so $\co_\mathfrak{p}^\times/\co_\mathfrak{p}^{\times m}$ is also finite.

    Since $K_\mathfrak{p}$ is the field of fractions of the discrete valuation ring $\co_\mathfrak{p}$, every element of $K_\mathfrak{p}^\times$ can be uniquely written as $\pi_\mathfrak{p}^nu$ with $n\in\Z$ and $u\in\co_\mathfrak{p}^\times$, where $\pi_\mathfrak{p}$ is a uniformizer of $\co_\mathfrak{p}$. Since $(\pi_\mathfrak{p}^n)^m\in K_\mathfrak{p}^{\times m}$, every coset in $K_\mathfrak{p}^\times/K_\mathfrak{p}^{\times m}$ contains a representative of the form $\pi_\mathfrak{p}^nu$, where $0\leq n\leq m-1$. Thus, the map 
    \[
    \{0,1,\dots,m-1\}\times \co_\mathfrak{p}^\times\rightarrow K_\mathfrak{p}^\times/K_\mathfrak{p}^{\times m}: (n,u)\mapsto \pi_\mathfrak{p}^nu\cdot K_\mathfrak{p}^{\times m}
    \]
    is surjective. Observe that if $u_1=u_2v^m$, where $u_1,u_2,v\in\co_\mathfrak{p}^\times$, then $\pi_\mathfrak{p}^nu_1=\pi_\mathfrak{p}^nu_2v^m$. Thus, the above surjection induces a well-defined surjection 
    \[
    \{0,1,\dots,m-1\}\times \co_\mathfrak{p}^\times/ \co_\mathfrak{p}^{\times m}\rightarrow K_\mathfrak{p}^\times/K_\mathfrak{p}^{\times m}.
    \]
    Since $\co_\mathfrak{p}^\times/ \co_\mathfrak{p}^{\times m}$ is finite, it follows that $K_\mathfrak{p}^\times/K_\mathfrak{p}^{\times m}$ is finite. Moreover, 
    \[
    |K_\mathfrak{p}^\times/K_\mathfrak{p}^{\times m}|\leq m |\co_\mathfrak{p}^\times/ \co_\mathfrak{p}^{\times m}|\leq m | \co_\mathfrak{p}^\times/1+\mathfrak{p}^{k_\mathfrak{p}}\co_\mathfrak{p}|=m |(\co_\mathfrak{p}/\mathfrak{p}^{k_\mathfrak{p}})^\times|.
    \]
    This proves (2).

    \medskip

    For (3), let us first show that $\bar{v}_\mathfrak{p}$ is well-defined. Suppose $\beta_1,\beta_2$ lie in the same class; then $\beta_1=\beta_2\gamma^m$ for some $\gamma\in K_\mathfrak{p}^\times$. Thus, 
    \[
    v_\mathfrak{p}(\beta_1)=v_\mathfrak{p}(\beta_2)+mv_\mathfrak{p}(\gamma)\equiv v_\mathfrak{p}(\beta_2)\pmod m.
    \]
    Therefore, $\bar{v}_\mathfrak{p}$ is well-defined. Let $[g]\in B_\mathfrak{p}$. By the definition of $B_\mathfrak{p}$, there exists a $0\neq \alpha\in [g]\cap \co_\mathfrak{p}$. Write $v_\mathfrak{p}(\alpha)=qm+r$, $0\leq r\leq m-1$. Consider the element $\tilde{\alpha}=\alpha \pi_\mathfrak{p}^{-qm}=\alpha (\pi_\mathfrak{p}^{-q})^{m}$. Then, we see that $\tilde{\alpha}\in [g]$, and $v_\mathfrak{p}(\tilde{\alpha})=v_\mathfrak{p}(\alpha)-qm=r\geq 0$, so $\tilde{\alpha}\in[g]\cap\co_\mathfrak{p}$. Thus, the set $\{v_\mathfrak{p}(\alpha)\mid \alpha\in [g]\cap \co_\mathfrak{p}\}$ is a non-empty subset of $\Z_{\geq 0}$ whose elements modulo $m$ are fixed, in particular $v_\mathfrak{p}(\alpha)\equiv v_\mathfrak{p}(\tilde{\alpha})\pmod m$ because they both lie in the same class $[g]$. Therefore, the minimum $r_g$ exists, is attained by $\tilde{\alpha}$, and $0\leq r_g\leq m-1$. This completes the proof.
\end{proof}

Lemma \ref{lem:8.3} gives an element in $[g]\cap \co_\mathfrak{p}$ whose valuation is $r_g$. We now show that, in fact, there exists a totally positive (global) integer whose valuation is $r_g$.

\begin{lemma}\label{lem:8.4}
    Let $K$ be a totally real number field of degree $d$, $\mathfrak{p}\subset\co_K$ be a prime ideal, and $[g]\in B_\mathfrak{p}$, where $B_\mathfrak{p}$ is as defined in Lemma \ref{lem:8.3}. Then, there exists an element $\eta_g\in\co_K^+$ such that $\eta_g\in[g]$ and $v_\mathfrak{p}(\eta_g)=r_g$, where $r_g$ is from Lemma \ref{lem:8.3}.
\end{lemma}

\begin{proof}
    Let $[g]\in B_\mathfrak{p}$. By Lemma \ref{lem:8.3} (3), there exists a $\alpha\in [g]\cap \co_\mathfrak{p}$ with $v_\mathfrak{p}(\alpha)=r_g$. We first approximate $\alpha$ by an element of $\co_K$ modulo a large power of $\mathfrak{p}$. By Lemma \ref{approximation}, for every $n\geq 1$, there is an isomorphism $\co_K/\mathfrak{p}^n\rightarrow \co_\mathfrak{p}/\mathfrak{p}^n\co_\mathfrak{p}$. Consequently, there exists $\beta\in \co_K$ such that 
    \[
    \beta\equiv\alpha\pmod {\mathfrak{p}^{r_g+k_\mathfrak{p}}\co_\mathfrak{p}},
    \]
    where $k_\mathfrak{p}=2v_\mathfrak{p}(m)+1$ is from Lemma \ref{lem:8.3}.

    Let $p$ be a rational prime below $\mathfrak{p}$, so $p\in \mathfrak{p}$. Then, the rational integer $q=p^{r_g+k_\mathfrak{p}}\in \mathfrak{p}^{r_g+k_\mathfrak{p}}$. In particular, for every $n\in\Z$, $nq\in\mathfrak{p}^{r_g+k_\mathfrak{p}}$. Thus, for all $n\in\Z$, $\eta_g:=\beta+nq$ still satisfies $\eta_g\equiv \beta\equiv\alpha\pmod {\mathfrak{p}^{r_g+k_\mathfrak{p}}\co_\mathfrak{p}}$. Set 
    \[
    c=\max_{1\leq i\leq d} |\sigma_i(\beta)|.
    \]
    Choose $n\in \Z$ so that $nq>c$. Then, for each $1\leq i\leq d$, $\sigma_i(\eta_g)=\sigma_i(\beta)+nq>nq-c>0$, so $\eta_g\in \co_K^+$. Now, we verify that $v_\mathfrak{p}(\eta_g)=r_g$.

    By definition of $\eta_g$, we have $\eta_g-\alpha=(\beta-\alpha)+nq$. By construction, $\eta_g-\alpha, \beta-\alpha+nq\in \mathfrak{p}^{r_g+k_\mathfrak{p}}\co_\mathfrak{p}$. Therefore, $v_\mathfrak{p}(\eta_g-\alpha)\geq r_g+k_\mathfrak{p}$. Consequently, we have $v_\mathfrak{p}(\eta_g-\alpha)\geq r_g+k_\mathfrak{p}>r_g=v_\mathfrak{p}(\alpha)$. Recall that for a non-Archimedean place, if $v(x-y)>v(x)$, then $v(x)=v(y)$. Thus, it follows that $v_\mathfrak{p}(\eta_g)=v_\mathfrak{p}(\alpha)=r_g$. It remains to show $\eta_g\in[g]$. 

    Write 
    \[
    \frac{\eta_g}{\alpha}=1+\frac{\eta_g-\alpha}{\alpha}.
    \]
    Then, 
    \[
    v_\mathfrak{p}\left(\frac{\eta_g-\alpha}{\alpha}\right)=v_\mathfrak{p}(\eta_g-\alpha)-v_\mathfrak{p}(\alpha)\geq (r_g+k_\mathfrak{p})-r_g=k_\mathfrak{p}.
    \]
    Thus, $\eta_g/\alpha\in 1+\mathfrak{p}^{k_\mathfrak{p}}\co_\mathfrak{p}\subseteq\co_\mathfrak{p}^{\times m}\subset K_\mathfrak{p}^{\times m}$ by Lemma \ref{lem:8.3} (1). Therefore, $\eta_g=\beta+nq\in [g]$ for all $n\in\Z$ with $nq>c$ as $\alpha\in [g]$ completing the proof.
\end{proof}

\begin{rem}
    In Lemma \ref{lem:8.4}, we state the existence of a single $\eta_g$, while we prove that there are many such $\eta_g$.  This minor inaccuracy is intentional, as our subsequent proofs only require a single $\eta_g$.
\end{rem}

Next, we prove a stronger version of this lemma, in which we require that these $\eta_g$ lie outside a given finite subset of $\co_K^+/\co_K^{\times m}$.

\begin{lemma}\label{lem:8.4'}
    Let the hypothesis be as in Lemma \ref{lem:8.4}, and $D\subset \co_K^+/\co_K^{\times m}$ be finite. Then, the element $\eta_g\in \co_K^+$ in Lemma \ref{lem:8.4} can be chosen such that the class of $\eta_g$ in $\co_K^+/\co_K^{\times m}$ does not lie in $D$.
\end{lemma}

\begin{proof}
    With the notations and definitions as in the proof of Lemma \ref{lem:8.4}, we have $\eta_g=\beta +nq\in \co_K^+\cap [g]$ for all $n\in \Z$ with $nq>c$, and $v_\mathfrak{p}(\eta_g)=r_g$. For such $n$, we have $\sigma_i(\eta_g)\geq nq-c$ for all $1\leq i\leq d$. Thus, 
    \[
    \N(\eta_g)\geq (nq-c)^d\rightarrow \infty,
    \]
    as $n\rightarrow \infty$. Since elements in a class of $\co_K^+/\co_K^{\times m}$ have same norm, and the set $D$ is finite; so choosing $n$ with $\N(\eta_g)>\max_{\gamma\in D}\N(\gamma)$ forces the class of $\eta_g$ to lie outside $D$. Therefore, taking $\eta_g$ for such $n$ completes the proof of the lemma.
\end{proof}

\begin{definition}\label{def:8.6}
    Let $\mathfrak{p}\subset \co_K$ be a prime ideal. A finite set $W_\mathfrak{p}\subset \co_K^+$ is a \emph{local criterion set at $\mathfrak{p}$} if for every $[g]\in B_\mathfrak{p}$, there is a $\eta_g\in \co_K^+$ with $\eta_g\in [g]$ and $v_\mathfrak{p}(\eta_g)=r_g$, where $B_\mathfrak{p}$ and $r_g$ are as in Lemma \ref{lem:8.3}. We denote the set 
    \begin{equation}\label{wp}
    W_\mathfrak{p}=\{\eta_g\in\co_K^+\mid [g]\in B_\mathfrak{p}\}.
\end{equation}
\end{definition}

\begin{rem}\label{rem1}
    Note that Lemma \ref{lem:8.4'} provides $W_\mathfrak{p}$ not meeting $D$ for any given finite subset $D\subset\co_K^+/\co_K^{\times m}$.
\end{rem}

Note that finiteness of $W_\mathfrak{p}$ follows from Lemma \ref{lem:8.3}.
In the following proposition, we clarify why we are calling $W_\mathfrak{p}$ a local criterion set.

\begin{proposition}\label{prop:8.5}
    Let $m\geq 2$ be an even integer, $K$ be a totally real number field, $\mathfrak{p}\subset\co_K$ a prime ideal, and $Q(x)=\sum_{i=1}^na_ix_i^m$ be a diagonal $m$-ic form over $K$. Assume that $Q$ represents all $W_\mathfrak{p}$ over $K$, i.e., for each $\eta_g\in W_\mathfrak{p}$, there exists a $x\in\co_K^n$ with $Q(x)=\eta_g$. Then, $Q$ represents every element of $\co_\mathfrak{p}$ over $\co_\mathfrak{p}$, i.e., for every $\alpha\in\co_\mathfrak{p}$, there is a $y\in\co_\mathfrak{p}^n$ such that $Q(y)=\alpha$.
\end{proposition}

\begin{proof}
    It is clear that $\alpha=0$ is always represented by $Q$ over $\co_\mathfrak{p}$. Therefore, we assume $\alpha\in \co_\mathfrak{p}\setminus\{0\}$. Let $[g]=\alpha\cdot K_\mathfrak{p}^{\times m}$ be the coset of $\alpha$. Since $\alpha\neq 0$, we have $[g]\in B_\mathfrak{p}$, where $B_\mathfrak{p}$ is from Lemma \ref{lem:8.3}. By Lemma \ref{lem:8.4}, there exists $\eta_g\in \co_K^+$ such that $\eta_g\in[g]$ and $v_\mathfrak{p}(\eta_g)=r_g$. By our assumption, there exists $x_g\in\co_K^n$ such that $Q(x_g)=\eta_g$. Since $\alpha,\eta_g\in[g]$, there exists a $\gamma\in K_\mathfrak{p}^\times$ such that $\alpha=\eta_g \gamma^m$. Then, we have $v_\mathfrak{p}(\alpha)=v_\mathfrak{p}(\eta_g)+mv_\mathfrak{p}(\gamma)$.

    Recall from proof of Lemma \ref{lem:8.3} that for  every $\alpha\in [g]\cap \co_\mathfrak{p}$, we have $v_\mathfrak{p}(\alpha)\equiv r_g \pmod m$ by minimality, so, we have $v_\mathfrak{p}(\alpha)=r_g+k_\alpha m$ for some $k_\alpha\in\Z_{\geq 0}$. Using the fact that $v_\mathfrak{p}(\alpha)=v_\mathfrak{p}(\eta_g)+mv_\mathfrak{p}(\gamma)$, we get $v_\mathfrak{p}(\gamma)=k_\alpha\geq 0$, so $\gamma\in\co_\mathfrak{p}$.

    Define $y=\gamma x_g\in\co_\mathfrak{p}^n$. Then, we have 
    \[
    Q(y)=\sum_{i=1}^na_i(\gamma x_{g,i})^m=\gamma^m \sum_{i=1}^na_i(x_{g,i})^m=\eta_g\gamma^m=\alpha,
    \]
    as we wanted.
\end{proof}

Note that Proposition \ref{prop:8.5} is still not sufficient for the assumption of Corollary \ref{cor:alg} because the multiplication by $\gamma$ (in the above proof) may enlarge the valuation depending on $\alpha$. However, in the next lemma, we will show that this issue can be resolved by adding one extra coefficient in the $m$-ic form.

\begin{lemma}\label{lem:8.6}
    Let $m\geq 2$ be an even integer, $K$ be a totally real number field, $\mathfrak{p}\subset\co_K$ be a prime ideal, and $Q(x)=\sum_{i=1}^na_ix_i^m$ be a totally positive definite diagonal $m$-ic form over $K$ such that 
    \[
    Q=Q_1\perp\langle \beta \rangle\perp Q_2,
    \]
    where $\beta \in \co_K^+$, $Q_1$ is a diagonal $m$-ic form representing all $W_\mathfrak{p}$ (defined in \eqref{wp}), and $Q_2$ is an arbitrary diagonal $m$-ic form (possibly empty form). Then, for every $\alpha\in\co_K^+$, there exists $x\in\co_\mathfrak{p}^n$ such that $Q(x)=\alpha$ and 
    \[
    \min_{1\leq i\leq n} v_\mathfrak{p}(m a_ix_i^{m-1})\leq v_\mathfrak{p}(m)+v_\mathfrak{p}(\beta).
    \]
\end{lemma}

\begin{proof}
    Let $\alpha\in\co_K^+$. Since $\beta\in \co_K^+$, the element $\gamma=\alpha-\beta\in \co_\mathfrak{p}$. By Proposition \ref{prop:8.5}, there exists $y=(y_1,\dots,y_k)\in \co_\mathfrak{p}^k$ such that $Q_1(y)=\gamma$, where $k$ is the rank of $Q_1$. Define $x=(y_1,\dots,y_k,1,0,\dots,0)\in \co_\mathfrak{p}^n$. Then we have 
    \[
    Q(x)=Q_1(y)+\beta \cdot 1^m+0=\gamma+\beta.
    \]
    For valuation estimate, observe that 
    \[
    \min_{1\leq i\leq n} v_\mathfrak{p}(m a_ix_i^{m-1})\leq v_\mathfrak{p}(m \beta \cdot 1^{m-1})\leq  v_\mathfrak{p}(m)+v_\mathfrak{p}(\beta),
    \]
    as required. 
\end{proof}

\section{Existence and uniqueness of criterion set}\label{S9}

In this section, we will prove that a unique minimal criterion set for the universal diagonal $m$-ic form exists. 

Throughout this section, $K$ will be a totally real number field of degree $d$ and $m\geq 2$ will be an even integer. We will utilize the notations and definitions from Section \ref{S3}.

\subsection{Termination of Algorithm \ref{alg:escalation}}

We now establish that the Algorithm \ref{alg:escalation} indeed terminates, which will imply, by Proposition \ref{lem:finitecriterionset}, that a criterion set exists. Recall from Section \ref{S3} that an escalation path is a sequence of diagonal $m$-ic forms $Q_0 (\text{empty form}), Q_1,Q_2,\dots$ with $Q_{i+1}=Q_i\perp\langle \beta_{i+1}\rangle$, where $\beta_{i+1}\in E(Q_i,\alpha_i)$ for some truant $\alpha_i\in T(Q_i)$, and each $Q_i$ has rank $i$. Note that a path is infinite if $T(Q_i)\neq \emptyset$; equivalently, none of the $Q_i$ are universal. The algorithm \ref{alg:escalation} yields a rooted tree $\mathcal{T}$, with the root (root node) $Q_0$. Let us call $\mathcal{T}$ the \emph{escalation tree}. By a node, we mean a vertex of the tree. For a non-universal form $Q$, the set $\{Q\perp \langle\beta \rangle\mid \beta\in  E(Q,\alpha_i)\}$ denotes the branches at the node $Q$. In particular, by Lemma \ref{finitenessof truant}, we know that each node has finitely many branches. 

\begin{lemma}\label{lem:9.2}
    Let $\mathcal{T}$ be the escalation tree. Then, the Algorithm \ref{alg:escalation} terminates if and only if $\mathcal{T}$ is a finite tree, i.e., the set of nodes of $\mathcal{T}$ is finite.
\end{lemma}

\begin{proof}
    Run Algorithm \ref{alg:escalation}; it maintains a queue $\mathcal{Q}$ of forms to process, starting with the empty form $Q_0$. In each iteration, it takes a form $Q$, computes its truant set $T(Q)$, and if $T(Q)\neq \emptyset$ inserts all $Q\perp \langle\beta\rangle$ with $\beta \in E(Q,\alpha)$ for each $\alpha\in T(Q)$, into $\mathcal{Q}$. The algorithm stops when $\mathcal{Q}$ is empty. Therefore, the forms ever inserted into $\mathcal{Q}$ are exactly the node set of $\mathcal{T}$. Thus, the algorithm terminates if and only if the node set of $\mathcal{T}$ is finite.
\end{proof}

Next, we need the well-known K\H{o}nig's lemma from graph theory. 

\begin{lemma}[{\cite[Theorem $1.10$]{kl}}]\label{koniglemma}
    Let $\mathcal{T}$ be an infinite rooted tree such that every node has finitely many branches. Then, there exists an infinite path through $\mathcal{T}$. 
\end{lemma}

\begin{theorem}\label{termination}
    Let $m\geq 2$ be an even integer, and $K$ be a totally real number field of degree $d$. Then, Algorithm \ref{alg:escalation} terminates after processing finitely many diagonal $m$-ic forms over $K$.
\end{theorem}

\begin{proof}
    By Lemma \ref{lem:9.2}, it is enough to show that the escalation tree $\mathcal{T}$ is finite. For contradiction, suppose that $\mathcal{T}$ is infinite. Then, by Lemma \ref{koniglemma}, there exists an infinite escalation path
    \[
    Q_0, Q_1,Q_2,\dots  \text{ with } Q_{i+1}=Q_i\perp\langle \beta_{i+1}\rangle, 
    \]
    where $\beta_{i+1}\in E(Q_i,\alpha_i)$ for some truant $\alpha_i\in T(Q_i)$. Note that, for each $i$, $\operatorname{rank}(Q_i)=i$, and $Q_i$ is not universal. By Lemma \ref{lem:3.8}, the truants $\alpha_1,\alpha_2,\dots$ are pairwise distinct, and $\min \mathrm{N}(\operatorname{Unrep}(Q_i))\rightarrow \infty$ as $i\rightarrow\infty$.

    Since the path is infinite, there exist $a_1,a_2,a_3$ leading three coefficients of $Q_i$ for all $i\geq 3$. By Proposition \ref{prop:8.2}, there exists a finite set $S$ of prime ideals in $\co_K$ (depending only on $m,a_1,a_2,a_3$) such that for every $i\geq 3$, every $\mathfrak{p}\notin S$, and every $\alpha \in \co_K^+$, there exists a $x\in\co_\mathfrak{p}^i$ with $Q_i(x)=\alpha$ and 
    \[
     \min_{1\leq j\leq i} v_\mathfrak{p}(m a_jx_j^{m-1})=0.
    \]

    For each $\mathfrak{p}\in S$, let $W_\mathfrak{p}$ be a local criterion set as defined in \eqref{wp}. Define 
    \[
    W=\bigcup_{\mathfrak{p}\in S} W_\mathfrak{p}, \quad N_W=\max_{\eta\in W}(\N(\eta)).
    \]
    Note that both $W$ and $N_W$ are finite because $S$ is finite and each $W_\mathfrak{p}$ is finite. As $\min\mathrm{N}(\operatorname{Unrep}(Q_i))\rightarrow \infty$, and $N_W$ is fixed, there exists an index $i_W\geq 3$ such that $\min \mathrm{N}(\operatorname{Unrep}(Q_i))>N_W$. By definition, every class in $\co_K^+/\co_K^{\times m}$ of norm $< N(Q_{i_W})$ is represented by $Q_{i_W}$. In particular, every $\eta\in W$ has $\N(\eta)\leq N_W<\min \mathrm{N}(\operatorname{Unrep}(Q_i))$, so $\eta$ is represented by $Q_{i_W}$, i.e., for each $\mathfrak{p}\in S$ and $\eta\in W$, there exists $y_\eta\in \co_K^{i_W}$ with $Q_{i_W}(y_\eta)=\eta$.

    Let $\beta_{i_W+1}\in E(Q_{i_W},\alpha_{i_W})$ be the next coefficient; such an element exists because the path is infinite. Now, we define
    \[
    \hat{i}=\max\{i_W+1, d(d+1)(m-1)2^{m-1}+1\},
    \]
    where $d$ is the degree of $K$, and set $\hat{Q}=Q_{\hat{i}}$. Then, $\operatorname{rank}(\hat{Q})=\hat{i}> d(d+1)(m-1)2^{m-1}$. By construction, the form $\hat{Q}$ admits an orthogonal decomposition 
    \[
    \hat{Q}=Q_{i_W}\perp \langle \beta_{i_W+1}\rangle\perp Q'
    \]
    for some diagonal $m$-ic form $Q'$ (possibly empty form).

    For $\mathfrak{p}\notin S$; since $\hat{i}\geq 3$, by Proposition \ref{prop:8.2}, for every $\alpha\in \co_K^+$, there exists $x\in \co_\mathfrak{p}^{\hat{i}}$ with $\hat{Q}(x)=\alpha$ and 
    \[
    \min_{1\leq i\leq \hat{i}} v_\mathfrak{p}(m a_ix_i^{m-1})=0.
    \]
    For $\mathfrak{p}\in S$, by Lemma \ref{lem:8.6}, for every $\alpha\in \co_K^+$, there exists $x\in \co_\mathfrak{p}^{\hat{i}}$ such that $\hat{Q}(x)=\alpha$ and 
    \[
    \min_{1\leq i\leq \hat{i}} v_\mathfrak{p}(m a_ix_i^{m-1})\leq v_\mathfrak{p}(m)+v_\mathfrak{p}(\beta_{i_W+1}).
    \]
    Therefore, the assumption of Corollary \ref{cor:alg} holds with 
    \[
    t=\max \left( 0, \{v_\mathfrak{p}(m)+v_\mathfrak{p}(\beta_{i_W+1})\mid \mathfrak{p}\in S\} \right)\geq 0.
    \]
    Note that $t$ is finite because $S$ is finite. Then, by Corollary \ref{cor:alg}, there exists $C>0$ depending only on $K,m,\hat{Q}, \text{ and }t$ such that for every $\alpha\in\co_K^+$ with $\N(\alpha)>C$ is represented by $\hat{Q}$ over $K$. By Lemma \ref{lem:3.8} (1), the same conclusion holds for all $Q_i$ with $i\geq \hat{i}$. Consequently, if for some $i\geq \hat{i}$, an element $\alpha\in \co_K^+/\co_K^{\times m}$ is not represented by $Q_i$, then $\N(\alpha)\leq C$. In particular, for each $i\geq \hat{i}$, the truant $\alpha_{i+1}\in T(Q_i)$ is not represented by $Q_i$, so $\N(\alpha_{i+1})\leq C$ for all $i\geq \hat{i}$. By Lemma \ref{lem:3.8} (2), the truants $\alpha_i$ are pairwise distinct. Thus, we have infinitely many distinct classes in $\alpha\in \co_K^+/\co_K^{\times m}$ with norm $\leq C$, a contradiction. Hence $\mathcal{T}$ is finite, and the algorithm terminates after processing finitely many diagonal $m$-ic forms over $K$.
\end{proof}

Now, we can make Proposition \ref{lem:finitecriterionset} unconditional by using Theorem \ref{termination}.

\begin{corollary}\label{cor:9.4}
    Let $m\geq 2$ be an even integer and $K$ a totally real number field. Then, 
    \[
    \mathcal{C}_K^{(m)} =\mathcal{C}(K,m)= \bigcup_{Q \text{ processed by Algorithm \ref{alg:escalation}}} T(Q)
    \]
    is a criterion set for totally positive definite diagonal $m$-ic forms over $K$, i.e., a totally positive definite diagonal $m$-ic form over $K$ is universal if and only if it represents all elements of $\mathcal{C}_K^{(m)}$.
\end{corollary}

\begin{proof}
    This follows from Proposition \ref{lem:finitecriterionset} and Theorem \ref{termination}.
\end{proof}

\subsection{Uniqueness of criterion set} Recall from Definition \ref{def:critical}, an element $\alpha\in \co_K^+/\co_K^{\times m}$ is a critical element if there exists a totally positive definite diagonal $m$-ic form $Q$ over $K$ such that $Q$ does not represent $\alpha$, and $Q$ represents all $(\co_K^+/\co_K^{\times m})\setminus \{\alpha\}$. We denote by $\overline{\mathcal{C}}_K^{(m)}$ the set of all critical elements. With all the machinery developed so far, we will now show that $\overline{\mathcal{C}}_K^{(m)}$ is a criterion set for diagonal $m$-ic forms. This implies that $\overline{\mathcal{C}}_K^{(m)}$ is indeed a unique minimal criterion set with respect to inclusion, because critical elements must belong to every criterion set, thereby proving our main theorem from the introduction. To this end, using the strategy of Kala--Kr\'{a}sensk\'{y}--Romeo \cite[Proposition $3.1$]{kkr}, we classify the critical elements with the tools we have developed thus far. Indeed, our proof of uniqueness differs from \cite{kkr}, as we adapt the proof of Theorem \ref{termination} to a slightly different setting. 

\begin{theorem}\label{uniqueness}
    Let $m\geq 2$ be an even integer and $K$ be a totally real number field. The set of all critical elements $\overline{\mathcal{C}}_K^{(m)}$ is a criterion set for diagonal $m$-ic forms over $K$. Hence, $\overline{\mathcal{C}}_K^{(m)}$ is the unique criterion set that is minimal with respect to inclusion. 
\end{theorem}

To prove this, we first need the following characterization of critical elements. 

\begin{proposition}\label{prop:9.6}
    Let $m\geq 2$ be an even integer and $K$ be a totally real number field. Then, $\alpha\in \co_K^+/\co_K^{\times m}$ is critical if and only if $\alpha$ is a truant of some totally positive definite diagonal $m$-ic form over $K$. 
\end{proposition}

\begin{proof}
    If $\alpha\in \co_K^+/\co_K^{\times m}$ is critical, then by definition $\alpha$ is a truant of some totally positive definite diagonal $m$-ic form over $K$.

    Conversely, suppose $\alpha$ is a truant of some totally positive definite diagonal $m$-ic form, say $Q_1$ over $K$. We shall construct a form that represents all $\alpha\in (\co_K^+/\co_K^{\times m})\setminus \{\alpha\}$, thereby proving the the criticality of $\alpha$. To do this, we define a sequence of escalated forms $Q_1,Q_2,Q_3, \dots$, starting from $Q_1$ as follows. Assume that $Q_i$ does not represent all $ (\co_K^+/\co_K^{\times m})\setminus \{\alpha\}$. Choose $\beta_{i+1}\in (\co_K^+/\co_K^{\times m})\setminus \{\alpha\}$ that is not represented by $Q_i$ and has minimal norm among all non-represented elements by $Q_i$; we can choose $\beta_{i+1}$ a truant of $Q_i$. Define $Q_{i+1}=Q_i\perp \langle\beta_{i+1}\rangle$. Note that this is precisely the escalation process, but restricted to $(\co_K^+/\co_K^{\times m})\setminus \{\alpha\}$ and by truant $\beta_{i+1}$ of $Q_i$. If there exists an $i\geq 1$ such that $Q_i$ represents every $(\co_K^+/\co_K^{\times m})\setminus \{\alpha\}$, then $Q_i$ is the desired form that will prove $\alpha$ is critical. 

    Let us first claim that no $Q_i$ represents $\alpha$. For $i=1$, the claim is trivially true, as $\alpha$ is a truant of $Q_1$. Assume $Q_i$ does not represent $\alpha$, and consider the next form $Q_{i+1}$. If $Q_{i+1}$ represented $\alpha$, then 
    \[
    \alpha=\gamma+\beta_{i+1}z^m
    \]
    with $\gamma$ represented by $Q_i$ and $z\in \co_K\setminus \{0\}$ because $\alpha$ is not represented by $Q_i$. Then, we have 
    \begin{equation}\label{e:9.1}
        \N(\alpha)=\N(\gamma+\beta_{i+1}z^m)\geq \N(\beta_{i+1}z^m)=\N(\beta_{i+1})\N(z)^m.
    \end{equation}
    Since $\beta_{i+1}\in (\co_K^+/\co_K^{\times m})\setminus \{\alpha\}$ was chosen to have minimal norm, and $\alpha\notin (\co_K^+/\co_K^{\times m})\setminus \{\alpha\}$, we have $\N(\beta_{i+1})\geq \N(\alpha)$ (otherwise $\beta_{i+1}$ would have been a truant of $Q_i$ with norm smaller than $\N(\alpha)$, contradicting the definition of truant). Thus, 
    \[
    \N(\alpha)\geq \N(\beta_{i+1})\N(z)^m\geq \N(\alpha)\N(z)^m. 
    \]
    This implies $\N(z)^m=1$, and all inequalities in \eqref{e:9.1} are indeed equalities. Therefore, for every embedding $\sigma$, we have 
    \[
    \sigma(\gamma)+\sigma(\beta_{i+1})\sigma(z)^m=\sigma(\beta_{i+1})\sigma(z)^m,
    \]
    so $\sigma(\gamma)=0$ for all $\sigma$. Therefore, $\gamma=0$, and $\alpha=\beta_{i+1}z^m$ with $\N(z)=1$, i.e., $z\in \co_K^\times$. Thus, $\alpha$ and $\beta_{i+1}$ lie in the same class modulo $\co_K^{\times m}$, contradicting $\beta_{i+1}\in (\co_K^+/\co_K^{\times m})\setminus \{\alpha\}$. Therefore, no $Q_i$ represents $\alpha$. To establish the proposition, we now need to show that the construction of forms terminates, i.e., there is a $j\geq 1$ such that the form $Q_j$ in the sequence represents all $\co_K^+/\co_K^{\times m}$. To prove this, we use the strategy of Theorem \ref{termination}.

    Suppose for contradiction that the process does not terminate, i.e., for all $i\geq 1$, the form $Q_i$ does not represent all $\co_K^+/\co_K^{\times m}$. Then, by Lemma \ref{lem:3.8} (2), $\beta_1,\beta_2,\dots$ are pairwise distinct. Also, from Lemma \ref{lem:3.8}, it follows that the sequence $\N(\beta_i)$ is non-decreasing and $\N(\beta_i)\rightarrow\infty$ as $i\rightarrow\infty$.

    The form $Q_3$ in our construction has rank $\geq 3$. Let us fix its first three coefficients $a_1,a_2,a_3$. Note that $a_1,a_2,a_3$ will then be the leading first three coefficients of all $Q_i$ with $i\geq 3$. Since $i\geq 3$, by Proposition \ref{prop:8.2}, there exists a finite set $S$ of prime ideal in $\co_K$ such that for every $i\geq 3$, every $\mathfrak{p}\notin S$, and every $\gamma \in \co_K^+$, there exists a $x\in\co_\mathfrak{p}^i$ with $Q_i(x)=\gamma$ and 
    \[
     \min_{1\leq j\leq i} v_\mathfrak{p}(m a_jx_j^{m-1})=0.
    \]

    For each $\mathfrak{p}\in S$, by Lemma \ref{lem:8.4'} with $D=\{\alpha\}$, there exists a local criterion set $W_\mathfrak{p}'$ at $\mathfrak{p}$ (see Remark \ref{rem1}). Set 
    \[
    W'=\bigcup_{\mathfrak{p}\in S}W_\mathfrak{p}'\quad \text{and } \, N_{W'}=\max_{\eta\in W'}\N(\eta).
    \]
    Since $\N(\beta_i)\rightarrow\infty$, there exists an index $i_{W'}\geq 3$ such that $\N(\beta_{i_{W'}+1})>N_{W'}$. By definition, every element in $(\co_K^+/\co_K^{\times m})\setminus\{\alpha\}$ of norm $<\N(\beta_{i_{W'}+1})$ is represented by $Q_{i_{W'}}$. Since $W'\subset (\co_K^+/\co_K^{\times m})\setminus\{\alpha\}$, and every $\eta\in W'$ has $\N(\eta)\leq N_{W'}$, the form $Q_{i_{W'}}$ represents every $\eta\in W'$ over $K$. Now, put
    \[
    \hat{i}=\max\{i_{W'}+1, d(d+1)(m-1)2^{m-1}+1\},
    \]
    and $\hat{Q}=Q_{\hat{i}}$. Then, $\operatorname{rank}(\hat{Q})=\hat{i}\geq d(d+1)(m-1)2^{m-1}$ and 
    \[
    \hat{Q}=Q_{i_{W'}}\perp \langle \beta_{i_{W'}+1}\rangle\perp \langle \beta_{i_{W'}+2}, \beta_{i_{W'}+3},\dots,\beta_{\hat{i}}\rangle.
    \]
    The same argument as in Theorem \ref{termination} yields that the assumption of the Corollary \ref{cor:alg} holds with 
    $$t=\max \left( 0, \{v_\mathfrak{p}(m)+v_\mathfrak{p}(\beta_{i_{W'}+1})\mid \mathfrak{p}\in S\} \right).$$ Then, by Corollary \ref{cor:alg}, there exists $C>0$ such that every $\gamma\in \co_K^+$ with $\N(\gamma)$ is represented by $\hat{Q}$, and thus by every $Q_i$ with $i\geq \hat{i}$. Therefore, we have $\N(\beta_i)\leq C$. Since $\beta_i$ are pairwise distinct, we have infinitely many distinct $\beta_i$ with norm $\leq C$, a contradiction. Hence, the process terminates, i.e., there exists a $j\geq 1$ such that $Q_j$ represents all $(\co_K^+/\co_K^{\times m})\setminus \{\alpha\}$, thereby proving criticality of $\alpha$.
\end{proof}

\begin{proof}[{Proof of Theorem \ref{uniqueness}}]
    Let us first show that if a form represent all  $\overline{\mathcal{C}}_K^{(m)}$, then it is universal. Suppose a totally positive definite diagonal $m$-ic form $Q$ represents all $\overline{\mathcal{C}}_K^{(m)}$. If $Q$ is not universal, then it has a truant $\alpha$. Then, by Proposition \ref{prop:9.6}, $\alpha\in\overline{\mathcal{C}}_K^{(m)}$, a contradiction.
    Next, observe that every criterion set (finite or infinite) contains $\overline{\mathcal{C}}_K^{(m)}$ by definition, i.e., $\overline{\mathcal{C}}_K^{(m)}$ is minimal with respect to inclusion. 

    We shall now show that $\overline{\mathcal{C}}_K^{(m)}$ is finite. By Corollary \ref{cor:9.4}, there is a criterion set $\mathcal{C}_K^{(m)}$, which is a finite set. Then, by our observation $\overline{\mathcal{C}}_K^{(m)}\subseteq \mathcal{C}_K^{(m)}$; hence $\overline{\mathcal{C}}_K^{(m)}$ is a finite set. 
\end{proof}

\subsection{Special critical elements}
By Theorem \ref{uniqueness}, we know that the set of all critical elements is a criterion set for diagonal $m$-ic forms over $K$, which is also minimal with respect to inclusion. In this section, we will show some specific elements that must lie in $\overline{\mathcal{C}}_K^{(m)}$. All the results in this direction follow from \cite[Section 4]{kkr} with minor modifications.

We say that $\alpha\in \co_K$ is $m$th powerfree if $\beta^m\mid \alpha$ for $\beta\in \co_K$ implies $\beta\in \co_K^\times$; in other words, $\alpha$ is not divisible by the $m$th power of a non-unit.

\begin{lemma}\label{lem:9.7}
    The set $\overline{\mathcal{C}}_K^{(m)}$ consists of $m$th powerfree elements. 
\end{lemma}

\begin{proof}
    Suppose $\alpha\in\co_K^+$ is not $m$th powerfree, i.e., there exists $\beta\in \co_K\setminus\co_K^\times$ such that $\alpha=\beta^m\gamma$ for some $\gamma\in \co_K^+$. Let $Q$ be a totally positive definite diagonal $m$-ic form that represents all $\co_K^+\setminus\{\alpha\}$. Then, in particular, $Q$ represents $\gamma$, say $Q(x)=\gamma$ for some $x\in\co_K^{\operatorname{rank}(Q)}$. For then, $Q(\beta x)=\beta^m\gamma=\alpha$ leads to a contradiction. Therefore, every form representing all $\co_K^+\setminus\{\alpha\}$ must also represent $\alpha$, i.e., $\alpha$ is not critical. 
\end{proof}

Even more specific critical elements are $m$th powerfree indecomposable elements.

\begin{theorem}
    The set of all $m$th powerfree indecomposables is contained in $\overline{\mathcal{C}}_K^{(m)}$. 
\end{theorem}

\begin{proof}
    Let $\alpha$ be an $m$th powerfree indecomposable. It is enough to show that $\alpha$ is a truant of some totally positive definite diagonal $m$-ic form, thanks to Proposition \ref{prop:9.6}. Let us consider the set of all elements in $\co_K^+/\co_K^{\times m}$ with norm $<\N(\alpha)$. There are only finitely many of them; let $\beta_1,\beta_2,\dots,\beta_n$ be their representatives. Now, consider the diagonal $m$-ic form 
    \[
    Q(x_1,x_2,\dots,x_n)=\sum_{i=1}^n \beta_ix_i^m;
    \]
    this is a totally positive form as $\beta_i\in \co_K^+$, and it represents every element with norm $<\N(\alpha)$. We now show that $Q$ does not represent $\alpha$. Suppose, for contradiction, that $Q(x)=\alpha$ for some $x\in \co_K^{\operatorname{rank}(Q)}$. If there exist distinct $x_i,x_j\neq 0$ and $Q(x)=\alpha$, then $\alpha$ is not indecomposable. Therefore, there exists a unique $x_i\neq 0$, but then $\alpha$ is not $m$th powerfree. Hence, $Q$ cannot represent $\alpha$, as we intended.
\end{proof}

Let us now show some rational integers inside $\overline{\mathcal{C}}_K^{(m)}$.

\begin{proposition}
    Let $K$ be a totally real number field and $m\geq 2$ an even integer. Then, the following statements are true.
    \begin{enumerate}
        \item $1\in \overline{\mathcal{C}}_K^{(m)}$.
        \item  $2\in \overline{\mathcal{C}}_K^{(m)}$ if and only if $2$ is $m$th powerfree. 
        \item Assume that $2,3$ are $m$th powerfree, and that $\sqrt{5}\notin K$. Then, $3\in \overline{\mathcal{C}}_K^{(m)}$.
    \end{enumerate}
\end{proposition}

\begin{proof}
    \begin{enumerate}
        \item The empty form $Q_0$ has a truant $1$. Then, by Proposition \ref{prop:9.6}, $1\in \overline{\mathcal{C}}_K^{(m)}$.

        \item If $2\in \overline{\mathcal{C}}_K^{(m)}$, then by Lemma \ref{lem:9.7} $2$ is $m$th powerfree. Conversely, assume $2$ is $m$th powerfree. By Proposition \ref{prop:9.6}, it is enough to show that $2$ is a truant of some form. Let $\alpha_1,\dots \alpha_n$ be a complete set of distinct representatives of $m$th powerfree classes of norm $<\N(2)$, and set 
        \[
        Q(x)=\sum_{i=1}^n\alpha_i x_i^m.
        \]
        First, note that by construction, $Q$ represents every class of norm $<\N(2)$. Suppose, if possible, that $Q$ represents $2$. By the classification of decompositions of $2$ \cite[Proposition $4.1$]{kkr}, we have only one non-trivial decomposition $2=1+1$. 

        If there exists a unique $i$ such that $\alpha_ix_i^m=2$, then $x_i$ must be a unit, as $2$ is $m$th powerfree; thus, $\alpha$ and $2$ lie in the same class, a contradiction. Therefore, there exists $i\neq j$ such that $\alpha_ix_i^m+\alpha_jx_j^m=2$. Then, $\alpha_ix_i^m=1=\alpha_jx_j^m$. This implies that $\alpha_i$ and $\alpha_j$ lie in the same class, a contradiction. Therefore, $2$ is a truant of $Q$.

        \item Assume that $2,3$ are $m$th powerfree, and that $\sqrt{5}\notin K$. Let $\alpha_1,\dots,\alpha_n$ be a set of distinct representatives of all classes of norm $<\N(3)$ except the class of $1,2$. Consider the $m$-ic form 
        \[
        Q(y_1,y_2,x_1,\dots,x_n)=y_1^m+y_2^m+\sum_{i=1}^n\alpha_ix_i^m.
        \]
        By construction, $Q$ represents all classes of norm $<\N(3)$. By the decomposition of $3$, we have 
        \[
        3=1+1+1=2+1=\left( \frac{1+\sqrt{5}}{2}\right)^2+\left( \frac{1-\sqrt{5}}{2}\right)^2.
        \]
        We do not consider the last decomposition as $\sqrt{5}\notin K$. Note that $y_i^m\neq 2$, as $2$ is $m$th powerfree, also $\alpha_ix_i^m\neq 2$ because the class of $2$ was excluded. Therefore, $3$ cannot be represented by $Q$ with decomposition $2+1$. As in (2), $3$ cannot be represented by $Q$ with decomposition $1+1+1$. Hence $3$ is a truant of $Q$. This completes the proof. \qedhere
    \end{enumerate}
\end{proof}

\bibliographystyle{amsalpha}
	
\bibliography{Citation}

\end{document}